\documentclass[11pt]{amsart}
\usepackage{amssymb,amscd,  latexsym, graphicx, mathrsfs, enumerate}

\newcommand{\nc}{\newcommand}
\numberwithin{equation}{section}
\newtheorem{theorem}{Theorem}[section]
\newtheorem{prop}[theorem]{Proposition}
\newtheorem{importnota}[theorem]{Important Notation}
\newtheorem{prblm}[theorem]{Problem}
\newtheorem{notation}[theorem]{Notation}
\newtheorem{caution}[theorem]{Caution}
\newtheorem{remark}[theorem]{Remark}
\newtheorem{lemma}[theorem]{Lemma}
\newtheorem{construction}[theorem]{Construction}
\newtheorem{corollary}[theorem]{Corollary}
\newtheorem{example}[theorem]{Example}
\newtheorem{conclusion}[theorem]{Conclusion}
\newtheorem{triviality}[theorem]{Triviality}
\newtheorem{proto}[theorem]{Prototype Quasifibration}
\newtheorem{cauex}[theorem]{Cautionary Example}
\newtheorem{propositiondef}[theorem]{Proposition-Definition}
\newtheorem{subth}{Nuisance}[theorem]
\newtheorem{ssubth}{ }[subth]
\newtheorem{conjecture}[theorem]{Conjecture}
\newtheorem{sidest}[theorem]{Side Story}
\newtheorem{miniexample}[theorem]{Example}
\theoremstyle{definition}
\newtheorem{defin}[theorem]{Definition}

\nc\tri[1]{\begin{triviality}}
\nc\side[1]{\begin{sidest}}
\nc\conj[1]{\begin{conjecture}}
\nc\prodef[1]{\begin{propositiondef}}
\nc\prt[1]{\begin{proto}}
\nc\lem[1]{\begin{lemma}}
\nc\sblm[1]{\begin{sublemma}}
\nc\pro[1]{\begin{prop}}
\nc\thm[1]{\begin{theorem}}
\nc\cor[1]{\begin{corollary}}
\nc\dfn[1]{\begin{defin}}
\nc\sthm[1]{\begin{subth}}
\nc\exm[1]{\begin{example}}
\nc\miniexm[1]{\begin{miniexample}}
\nc\plm[1]{\begin{prblm}}
\nc\rmk[1]{\begin{remark}}
\nc\subrmk[1]{\begin{subremark}}
\nc\ntn[1]{\begin{notation}}
\nc\cau[1]{\begin{caution}}
\nc\imn[1]{\begin{importnota}}
\nc\cax[1]{\begin{cauex}}
\nc\con[1]{\begin{construction}}
\nc\ssthm[1]{\begin{ssubth}}
\nc\cnc[1]{\begin{conclusion}}
\nc\elem{\end{lemma}}
\nc\esblm{\end{sublemma}}
\nc\eside{\end{sidest}}
\nc\econj{\end{conjecture}}
\nc\eprodef{\end{propositiondef}}
\nc\eprt{\end{proto}}
\nc\ethm{\end{theorem}}
\nc\ecor{\end{corollary}}
\nc\edfn{\end{defin}}
\nc\esthm{\end{subth}}
\nc\epro{\end{prop}}
\nc\etri{\end{triviality}}
\nc\eexm{\end{example}}
\nc\eminiexm{\end{miniexample}}
\nc\ermk{\end{remark}}
\nc\subermk{\end{subremark}}
\nc\eplm{\end{prblm}}
\nc\ecau{\end{caution}}
\nc\ecax{\end{cauex}}
\nc\eimn{\end{importnota}}
\nc\entn{\end{notation}}
\nc\econ{\end{construction}}
\nc\ecnc{\end{conclusion}}
\nc\essthm{\end{ssubth}}

\newcommand{\C}{\mathbb{C}}
\newcommand{\R}{\mathbb{R}}
\newcommand{\Q}{\mathbb{Q}}

\newcommand{\Z}{\mathbb{Z}}

\newcommand{\X}{\mathfrak{X}}

\newcommand{\A}{\mathbb{A}}

\renewcommand{\O}{\mathcal{O}}

\newcommand{\diag}{{\rm diag}}
\newcommand{\G}{\Gamma}

\renewcommand{\o}{\mathfrak{o}}
\newcommand{\ds}{\displaystyle}

\newcommand{\tr}{{\rm Tr}}

\newcommand{\lra}{\longrightarrow}

\newcommand{\ba}{\backslash}

\newcommand{\disc}{{\rm disc}}

\renewcommand{\Bbb}{\mathbb}

\title[Cusp forms for exceptional group of type $E_{7}$]
{Cusp forms for exceptional group of type $E_{7}$}
\author{Henry H. Kim and Takuya Yamauchi}
\thanks{The first author is partially supported by NSERC. The second author
is partially supported by JSPS Grant-in-Aid for Scientific Research No.23740027 and JSPS Postdoctoral 
Fellowships for Research Abroad No.378.}
\subjclass[2010]{}
\address{Henry H. Kim \\
Department of mathematics \\
 University of Toronto \\
Toronto, Ontario M5S 2E4, CANADA \\
and Korea Institute for Advanced Study, Seoul, KOREA}
\email{henrykim@math.toronto.edu}

\address{Takuya Yamauchi \\
Department of mathematics, Faculty of Education\\
Kagoshima University\\
Korimoto 1-20-6 Kagoshima 890-0065, JAPAN and 
Department of mathematics \\
 University of Toronto \\
Toronto, Ontario M5S 2E4, CANADA}
\email{yamauchi@edu.kagoshima-u.ac.jp or tyama@math.toronto.edu}

\begin{document}
\begin{abstract}
Let $\bold{G}$ be the connected reductive group of type $E_{7,3}$ over $\Q$ and $\frak T$ be the corresponding 
symmetric domain in $\C^{27}$. Let $\Gamma=\bold{G}(\Z)$ be the arithmetic subgroup defined by Baily. In this paper, for any positive integer $k\ge 10$, we will construct a (non-zero) holomorphic cusp form on $\frak T$ of weight $2k$ with respect to $\G$ from 
a Hecke cusp form in $S_{2k-8}(SL_2(\Z))$. This lift is an analogue of Ikeda's construction 
(\cite{Ik2},\cite{Ik3},\cite{Yamana}).  
\end{abstract}
\maketitle
\tableofcontents

\section{Introduction}
Let $\bold{G}$ be  the exceptional Lie group  of type $E_{7,3}$ over $\Q$ and $\frak T\subset \C^{27}$ 
the corresponding bounded symmetric domain.  
The purpose of this paper is to construct holomorphic cusp forms on $\frak T$ from cusp forms for $SL_2$ over $\Q$. 
In \cite{Ik1}, Ikeda originally gave a (functorial) construction of a Siegel cusp form for $Sp_{2n}$ (rank $2n$) from a normalized Hecke eigenform on the upper half-plane $\mathbb{H}$ with respect to $SL_2(\Z)$ which has been conjectured by Duke and Imamoglu 
(Independently Ibukiyama formulated a conjecture in terms of Koecher-Maass series). He made use of the uniform property of the Fourier 
coefficients of Siegel Eisenstein series for $Sp_{2n}$ over $\Q$ and together with various deep facts established in \cite{Ik1} 
to prove Duke-Imamoglu conjecture. 
After this work, his construction was generalized to unitary groups $U(n,n)$ (\cite{Ik2}), 
quaternion unitary groups $Sp(n,n)$ (\cite{Yamana}), and symplectic groups $Sp_{2n}$ over totally real fields (\cite{Ik4},\cite{Ik&H}). 
Historically, in the case of $Sp_2$, the resulting cusp form is called Saito-Kurokawa lift which has been studied thoroughly 
(\cite{kuro}, \cite{ps}, \cite{cog&ps}). Our method follows his construction. The main obstruction is the hugeness of 
$E_{7,3}$. In aforementioned works, the theory of Jacobi forms has been understood well since the Heisenberg group inside the group in consideration is easy to handle. On the other hand, 
much less is known in the case of $E_{7,3}$. Therefore we have to consider a suitable Heisenberg subgroup in $E_{7,3}$ which has not been studied. To do this we analyze it in terms of roots. 

We now explain our main theorem. We refer the next section for the several notations which appear below.  
Let $\Gamma=\bold{G}(\Z)$ be the arithmetic subgroup defined by Baily in \cite{B} which is 
constructed by using the integral Cayley numbers $\frak o$. 
For a positive integer $k\ge 10$, let $E_{2k}$ be the Siegel Eisenstein series on $\frak T$ of 
weight $2k$ with respect to $\Gamma$. Then it has the Fourier expansion of form 
\begin{eqnarray*}
E_{2k}(Z) &=& \sum_{T\in \frak J(\Z)_+} a_{2k}(T) \exp(2\pi\sqrt{-1}(T,Z)),\ Z\in\frak T,\\
a_{2k}(T) &=& C_{2k}\det(T)^{\frac{2k-9}{2}}\prod_{p|\det(T)} \widetilde{f}^p_T(p^{\frac{2k-9}{2}}),
\end{eqnarray*}
where $C_{2k}=2^{15}\displaystyle\prod_{n=0}^2 \frac {2k-4n}{B_{2k-4n}}$, and 
$\widetilde{f}^p_T(X)$ is a Laurent polynomial over $\Q$ in $X$ which is depending only on $T$ and $p$.  

Let $S_{2k-8}(SL_2(\Z))$ be the space of elliptic cusp forms of weight $2k-8\geq 12$ with respect to $SL_2(\Z)$.  
For each normalized Hecke eigenform $f=\ds\sum_{n=1}^\infty c(n)q^n,\ q=\exp(2\pi \sqrt{-1}\tau),\ \tau\in \mathbb{H}$ in $S_{2k-8}(SL_2(\Z))$ and each rational prime $p$, 
we define the Satake $p$-parameter $\alpha_p$ by $c(p)=p^{\frac{2k-9}{2}}(\alpha_p+\alpha^{-1}_p)$. 
For such $f$, consider the following formal series on $\frak T$: 
$$F(Z)=\sum_{T\in \frak J(\Z)_+} A(T)\exp(2\pi\sqrt{-1}(T,Z)),\ Z\in \frak T, \quad 
A(T)=\det(T)^{\frac{2k-9}{2}} \prod_{p|\det(T)} \widetilde{f}_T^p(\alpha_p).
$$
Then we will show 
\begin{theorem}\label{main-thm1} The function $F(Z)$ is a non-zero Hecke eigen cusp form on $\frak T$ of weight $2k$ with respect to 
$\G$. 
\end{theorem}
If $f$ has integer Fourier coefficients, then $F$ also has integer Fourier coefficients (Remark \ref{integer}).
By virtue of Theorem \ref{main-thm1}, $F=F(Z)$ gives rise to a cuspidal automorphic representation 
$\pi_F=\pi_\infty\otimes \otimes'_p\pi_p$ 
of $\bold G(\A)$. Then $\pi_\infty$ is a holomorphic discrete series of the lowest weight $2k$ associated to $-2k\varpi_7$ in the notation of \cite{Bour} (cf. \cite{knapp}, page 158).
For each prime $p$, $\pi_p$ is unramified. In fact, $\pi_p$ turns out to be a degenerate principal series 
${\rm Ind}_{\bold{P}(\Q_p)}^{\bold{G}(\Q_p)}\: |\nu(g)|^{2s_p}$, where $p^{s_p}=\alpha_p$. 
Then for each local component $\pi_p$, one can associate the local $L$-factor $L(s,\pi_p,St)$ of the standard $L$-function of $\pi_F$
 by using the Langlands-Shahidi method. 
Put $L(s,\pi_F,St)=\ds\prod_p L(s,\pi_p,St)$ and let 
$L(s,\pi_f)=\prod_p (1-\alpha_p p^{-s})(1-\alpha_p^{-1} p^{-s})$ be the 
automorphic $L$-function of the cuspidal representation $\pi_f$ attached to $f$. Then

\begin{theorem}\label{main-thm2} The degree $56$ standard L-function $L(s,\pi_F,St)$ of $\pi_F$ is given by
$$L(s,\pi_F,St)=L(s,{\rm Sym}^3 \pi_f)L(s,\pi_f)^2 \prod_{i=1}^4 L(s\pm i,\pi_f)^2 \prod_{i=5}^8 L(s\pm i, \pi_f),
$$
where $L(s,{\rm Sym}^3 \pi_f)$ is the symmetric cube $L$-function.
\end{theorem}

This paper is organized as follows. In Section 2, we fix notations on Cayley numbers and exceptional Jordan algebras and review their properties. In Section 3, we review the exceptional group of type $E_{7,3}$ and prove many facts which are not available in the literature. In Section 4, we define the Jacobi group inside the exceptional group using the root subgroups, and recall Weil representations and theta functions.
In Section 5, we review modular forms on the exceptional domain and define Jacobi forms of matrix indices and study the Fourier-Jacobi coefficients of a modular form
both in classical setting and in adelic setting. In Section 6, we review the result of M. Karel on Fourier coefficients of Eisenstein series and interpret Eisenstein series in terms of degenerate principal series, following \cite{Ku}. Section 7 is the main technical part, where we prove the analogue of Ikeda's result \cite{Ik3}, namely, the Fourier-Jacobi coefficients of Eisenstein series are a sum of products of theta functions and Eisenstein series. 
In Section 9, by following Ikeda \cite{Ik1},\cite{Ik2}, we construct a holomorphic cusp form on the exceptional group of type $E_{7,3}$.
Our situation is similar to unitary group case, in that we do not need to consider half-integral modular forms. 
In Section 10, we review the Hecke operators from Karel's thesis \cite{Ka1} and modify it to fit into representation theory. 
Then we prove that our cusp form is a Hecke eigenform with respect to this modified action. The degree $56$ standard $L$-function helps us to speculate on the Arthur parameter of $\pi_F$. We make a brief remark on it at the end of Section 11.
In the Appendix, we compute the discriminant of some quadratic forms and prove the orthogonal relation of theta functions we need. 

\medskip

\noindent\textbf{Acknowledgments.} We would like to thank T. Ibukiyama, T. Ikeda, H. Katsurada, M. Nevins, K. Takase, and S. Yamana for helpful discussions. We thank the referees for helpful remarks and encouragement. 

\section{Cayley numbers and exceptional Jordan algebras}\label{Cayley}
In this section we will recall the Cayley numbers and the exceptional Jordan algebras. We refer \cite{B},\cite{Coxeter}, and \cite{kim}. 
For any field $K$ whose characteristic is different from $2$ and $3$, 
the Cayley numbers $\frak C_K$ over $K$ is an eight-dimensional vector space over $K$ with basis 
$\{e_0=1,e_1,e_2,e_3,e_4,e_5,e_6,e_7\}$ satisfying the following rules for multiplication:
\begin{enumerate}
\item  $xe_0=e_0x=x$ for all $x\in\frak C_K$,
\item $e^2_i=-e_0$ for $i=1,\ldots, 7$,
\item $e_ie_{i+1}e_{i+3}=-e_0$ for any $i$ (mod 7).  
\end{enumerate}
For each $x=\ds\sum_{i=0}^7x_ie_i\in \frak C_K$, 
the map $x\mapsto \bar{x}=x_0e_0-\ds\sum_{i=1}^7x_ie_i$ defines an anti-involution of $\frak C_K$.  
The trace and the norm on $\frak C_K$ are defined by 
$$\tr(x):=x+\bar{x}=2x_0,\ N(x):=x\bar{x}=\sum_{i=0}^7x^2_i.
$$
The Cayley numbers $\frak C_K$ is neither commutative nor associative. In spite of this, we have 
$$\tr(xy)=\tr(yx),\ \tr(x\bar{y})=\tr(\bar{x}y),\ \tr((xy)z)=\tr(x(yz)).$$
We denote by $\frak o$, the integral Cayley numbers 
which is a $\Z$-submodule of $\frak C_K$ given by the following basis:
$$\alpha_0=e_0,\ \alpha_1=e_1,\ \alpha_2=e_2,\ \alpha_3=-e_4,\ 
\alpha_4=\frac{1}{2}(e_1+e_2+e_3-e_4),\ \alpha_5=\frac{1}{2}(-e_0-e_1-e_4+e_5),$$
$$\alpha_6=\frac{1}{2}(-e_0+e_1-e_2+e_6),\ \alpha_7=\frac{1}{2}(-e_0+e_2+e_4+e_7).$$
As shown in \cite{Coxeter}, $\frak o$ is stable under the operations of the anti-involution, multiplication, and 
addition. Further we have $\tr(x),\ N(x)\in \Z$ if $x\in \frak o$. By using this integral structure, 
for any $\Z$-algebra $R$, one can consider $\frak C_R=\frak o\otimes_\Z R$. 

Let $\frak J_K$ be the exceptional Jordan algebra consisting of the element: 
\begin{equation}\label{x}
X=(x_{ij})_{1\le i,j\le 3}=\left(\begin{array}{ccc}
a& x & y \\
\bar{x} & b & z \\
\bar{y}& \bar{z} & c
\end{array}\right),
\end{equation}
where $a,b,c\in Ke_0=K$ and $x,y,z\in \frak C_K$. 
In general, the matrix multiplication $X\cdot Y$ for two elements $X,Y\in \frak J_K$ does not belong to 
$\frak J_K$, but the square $X^2=X\cdot X$ always does. The composition of $\frak J_K$ is given by 
$$X\circ Y=\frac{1}{2}(X\cdot Y+Y\cdot X).$$ 
For the above $X$, we define the trace by $\tr(X):=a+b+c$,
and define an inner product on $\frak J_K\times \frak J_K$ by $(X,Y):=\tr(X\circ Y).$
Moreover we define 
$$\det(X):=abc-aN(z)-bN(y)-cN(x)+\tr((xz)\bar{y})$$
and a symmetric tri-linear form $(\ast,\ast,\ast)$ on $\frak J_K\times \frak J_K\times \frak J_K$ 
by 
$$
(X,Y,Z):=\frac{1}{6}\{\det(X+Y+Z)-\det(X+Y)-\det(Y+Z)-\det(Z+X)+\det(X)+\det(Y)+\det(Z)\}.
$$  
Then we define a bilinear pairing $\frak J_K\times \frak J_K\lra \frak J_K, (X,Y)\mapsto X\times Y$ by requiring 
the identity 
$$3(X,Y,Z)=(X\times Y,Z)=\tr((X\times Y)\circ Z)\ {\rm for\ any\ }Z\in \frak J_K.$$
In particular, for $X_i,\ i=1,2$ with entries as in (\ref{x}), we have 
\begin{equation}
X_1\times X_2=\left(\begin{array}{ccc}
\frac{b_1c_2+c_1b_2}{2}-\frac{\overline{z}_1z_2+z_2\overline{z}_1}{2}& A & B \\
\overline{A} & \frac{a_1c_2+c_1a_2}{2}-\frac{\overline{y}_1y_2+y_2\overline{y}_1}{2} & C \\
\overline{B} & \overline{C} & \frac{a_1b_2+b_1a_2}{2}-\frac{\overline{x}_1x_2+x_2\overline{x}_1}{2}
\end{array}\right)
\end{equation} 
where $A=\ds\frac{-c_1x_2-c_2x_1}{2}+\frac{y_1\overline{z}_2+y_2\overline{z}_1}{2},\ 
B=\frac{-b_1y_2-b_2y_1}{2}+\frac{x_1z_2+x_2z_1}{2},\ $ and 
$C=\ds\frac{-a_1z_2-a_2z_1}{2}+\frac{\overline{x}_1y_2+\overline{x}_2y_1}{2}$. 
By using integral Cayley numbers, we define a lattice
$$\frak{J}(\Z):=\{X=(x_{ij})\in \frak J_\Q\ |\ \text{$x_{ii}\in \Bbb Z$, and $x_{ij}\in \frak o$ for $i\ne j$} \},
$$
and put $\frak{J}(R)=\frak{J}(\Z)\otimes_\Z R$ for any $\Z$-algebra $R$. 
Although the composition ``$\circ$" does not preserve the integral structure, but 
the inner product $(\ast,\ast)$ does. Hence $(\frak{J}(R),\frak{J}(R))\in R$. 
Then one can show that the lattice $\frak J(\Z)$ in $\frak J_\Q$ is the self-dual 
with respect to $(\ast,\ast)$, namely 
$$\widetilde{\frak J(\Z)}:=\{X\in \frak J_\Q\ |\ (X,Y)\in \Z\ {\rm for\ all}\ Y\in \frak J(\Z)\}=\frak J(\Z).$$  
We also define $\frak{J}_2(R)$ as the set of all matrices of forms  
$$X=\left(\begin{array}{cc}
a& x  \\
\bar{x} & b 
\end{array}\right),\ a,b\in R,\ x\in \frak C_R.$$
Similarly we define the inner product on $\frak{J}_2(R)\times \frak{J}_2(R)$ 
by $(X,Y):=\frac 12\tr(XY+YX)$. 
For any such $X$, we define $\det(X):=ab-N(x).$
For $X$ as above, $r\in R$, and $\xi=\left(\begin{array}{c}
\xi_1 \\
\xi_2
\end{array}\right),\ \xi_i\in \frak C_R\ (i=1,2)$, it is easy to see that 
\begin{equation}\label{formula}
\det\left(\begin{array}{cc}
X & X\xi \\
{}^t\bar{\xi}X & r
\end{array}\right)=
\det(X)(r-{}^t\bar{\xi}X\xi)=\det(X)(r-(X,{}^t\bar{\xi}\xi))
\end{equation}
which will be used later (Section \ref{construction}). 
Henceforth we identify $\frak J_2(R)$ with a subspace of $\frak J(R)$ by 
$\left(\begin{array}{cc}
a& x  \\
\bar{x} & b 
\end{array}\right)\mapsto \left(\begin{array}{ccc}
a& x & 0 \\
\bar{x} & b & 0 \\
0& 0 & 0
\end{array}\right).$

We define 
$$R_3(K)=\{X\in \frak J_K\ |\ \det(X)\not=0 \}
$$ 
and define the set $R^+_3(K)$ consisting of squares of elements in $R_3(K)$. 
It is known that $R^+_3(\R)$ is an open, convex cone in $\frak J_\R$. 
We denote by $\overline{R^+_3(\R)}$ the closure of $R^+_3(\R)$ in $\frak J_\R\simeq \R^{27}$ 
with respect to Euclidean topology. For any subring $A$ of $\R$, set 
$$\frak J(A)_+:=\frak J(A)\cap R^+_3(\R),\quad \frak J(A)_{\ge 0}:=\frak J(A)\cap \overline{R^+_3(\R)}.
$$
We also define  
$$\frak J_2(A)_{+}=\Bigg\{\left(\begin{array}{cc}
a& x  \\
\bar{x} & b 
\end{array}\right)\in\frak J_2(A) \ \Bigg|\ a,b\in A\cap \R_{>0},\ ab-N(x)>0  \Bigg\},$$
and  
$$\frak J_2(A)_{\ge 0}=\Bigg\{\left(\begin{array}{cc}
a& x  \\
\bar{x} & b 
\end{array}\right)\in \frak J_2(A)\ \Bigg|\ a,b\in A\cap \R_{\ge 0},\ ab-N(x)\ge 0  \Bigg\}.$$
We define the exceptional domain as follows:
$$\frak T:=\{Z=X+Y\sqrt{-1}\in \frak J_\C\ |\ X,Y\in \frak J_\R,\ Y\in R^+_3(\R)\}$$
which is a complex analytic subspace of $\C^{27}$ . We also define 
$$\frak T_2:=\{X+Y\sqrt{-1}\in \frak J_2(\C)\ |\ X,Y\in \frak J_2(\R),\ Y\in \frak J_2(\R)_+\}.
$$

\section{Exceptional group of type $E_{7,3}$}\label{E7}
In this section we recall the exceptional group of type $E_{7,3}$. Put $\frak J=\frak J_K$ where $K$ is a field 
whose characteristic is different from $2$ and $3$. 
Define two subgroups of $GL(\frak J)$ by
\begin{eqnarray*}
\bold M &=&\{g\in GL(\frak J)\ |\, \det(gX)=\nu(g)\det(X), \text{for $\nu(g)\ne 0$}\}\\
\bold M' &=&\{g\in \bold M\ |\, \nu(g)=1\}.
\end{eqnarray*}

Then $\bold M$ is an algebraic group over $\Q$ of type $GE_{6}$, and $\bold M'$ is the derived group of $\bold M$, which is a simple group of type $E_{6,2}$. The center of $\bold M'$ is the group of cube roots of unity.

There is an automorphism $g\longmapsto g^*$ of $\bold M$ of order 2 by the identity 
\begin{equation}\label{identity}
(gX, g^*Y)=(g^* X, gY)=(X,Y).
\end{equation}
 Then $g^*$ is the inverse adjoint of $g$. It satisfies $g(X\times Y)=(g^*X)\times (g^*Y).$

Let $\bold G$ be the algebraic group over $\Q$ as in \cite{B}: Let $\bold X, \bold X'$ be two $K$-vector spaces, each isomorphic to $\frak J$, and $\Xi, \Xi'$ be copies of $K$. Let $\bold W=\bold X\oplus \Xi\oplus \bold X'\oplus \Xi'$, and for $w=(X,\xi,X',\xi')\in\bold W$,
define a quartic form $Q$ on $\bold W$ by 
$$Q(w)=(X\times X, X'\times X')-\xi \det(X)-\xi' \det(X')-\frac 14( (X,X')-\xi\xi')^2,
$$
and a skew-symmetric bilinear form $\{\,,\,\}$ by
$$\{w_1,w_2\}=(X_1,X_2')-(X_2,X_1')+\xi_1\xi_2'-\xi_2\xi_1'.
$$
Then 
$$\bold G(K)=\{ g\in GL(\bold W_K) |\, Qg=Q,\, g\{\,,\,\}=\{\,,\,\}\}.
$$
This defines a connected algebraic $\Q$-group of type $E_{7,3}$; The center of $\bold G(\R)$ is $\{\pm \text{id}\}$
 and the quotient of $\bold G(\R)$ by its center is the group of holomorphic automorphisms of $\frak T$. The real rank of $\bold G$ is 3, and it is split over $\Bbb Q_p$ for any prime $p$.

The group $\bold M$ can be considered as a subgroup of $\bold G$ by defining the action
$$g(X,\xi,X',\xi')=(gX,\nu(g)\xi,g^*X',\nu(g)^{-1}\xi').
$$
Let $\bold N$ be the subgroup of all transformations $p_B$ for $B\in \frak J$ as in \cite{B}. Recall the definition.
$$
p_B\begin{pmatrix} X\\ \xi\\X'\\ \xi'\end{pmatrix}=\begin{pmatrix} X+\xi' B\\ \xi+(B,X')+(B\times B,X)+\xi' \det(B)\\ X'+2B\times X+\xi' B\times B\\ \xi'\end{pmatrix}.
$$
The relative root system of $\bold G$ over $\Q$ is of type $C_3$, and we denote the positive roots by
$\{e_1\pm e_2, e_1\pm e_3, e_2\pm e_3, 2e_1, 2e_2, 2e_3\}$, and let $\Delta=\{e_1-e_2, e_2-e_3, 2e_3\}$ be the set of simple roots.
We describe their root spaces: For a positive root $\alpha$, let $\bold U_{\alpha}$ be the root subspace. For $1\leq i\leq j\leq 3$, let $e_{ij}$ is the $3\times 3$ matrix with a 1 in the intersection of the $i$-th row and $j$-th column and zeros elsewhere, and let $e_i=e_{ii}$. 
Then for $a,b,c\in K$, $x,y,z\in \frak C_K$,
\begin{eqnarray*}
&& \bold U_{2e_1}=\{p_{a e_1}\},\quad \bold U_{2e_2}=\{p_{a e_2}\},\quad \bold U_{2e_3}=\{p_{a e_3}\}\\
&& \bold U_{e_1+e_2}=\{p_{x e_{12}}\},\quad \bold U_{e_1+e_3}=\{p_{y e_{13}}\},\quad \bold U_{e_2+e_3}=\{p_{z e_{23}}\}\\
&& \bold U_{e_1-e_2}=\{m_{\bar x e_{21}}\in GL(\frak J):\, m_{\bar x e_{21}}X=(I+x e_{12})X(I+\bar x e_{21}),\, X\in \frak J\}\\
&& \bold U_{e_1-e_3}=\{m_{\bar y e_{31}}\in GL(\frak J):\, m_{\bar y e_{31}}X=(I+y e_{13})X(I+\bar y e_{31}),\, X\in \frak J\}\\
&& \bold U_{e_2-e_3}=\{m_{\bar z e_{32}}\in GL(\frak J):\, m_{\bar z e_{32}}X=(I+z e_{23})X(I+\bar z e_{32}),\,  X\in \frak J\}
\end{eqnarray*}

\begin{remark} Note that we are using different ordering of roots from \cite{B}. In \cite{B}, $\bold N$ consists of root spaces of negative non-compact roots. However, it is more convenient to make it correspond to positive roots so that it may correspond to the upper triangular matrices of the form $\begin{pmatrix} I_n&B\\O_n&I_n\end{pmatrix}$ in $Sp_{2n}$ case.
\end{remark}

Note the following 
\begin{equation*}
\quad m_{x e_{ij}}^*=m_{-\bar x e_{ji}}.
\end{equation*}
Let $\bold H$ be the group generated by $\bold U_{2e_3}$ and $\iota_{e_3}$, where
$\iota_{e_i}$ is the Weyl group element of $2e_i$, which is given by
$\iota_{e_i}=p_{e_i} p_{-e_i}' p_{e_i}$, where $p_{e_i}'$ generates the root subspace of $-2e_i$. Then $\bold H\simeq SL_2$.
Let $\iota=\iota_{e_1}\iota_{e_2}\iota_{e_3}$. Then $\iota^{-1}=-\iota$, and $p_B'=\iota p_{-B}\iota^{-1}$ will generate the opposite unipotent subgroup $\overline{\bold N}$ of $\bold N$. This $\iota$ plays the role of 
$\begin{pmatrix} 0&I_n\\-I_n&0\end{pmatrix}$ is $Sp_{2n}$. Its action is given by
$$\iota(X,\xi,X',\xi')=(-X',-\xi',X,\xi).
$$
We define two maximal parabolic $\Q$-subgroups: 
$$\bold P=\bold M\bold N,\quad \bold Q=\bold L\bold V,
$$
where $\bold V$ is generated by $\bold U_{\alpha}$ for $\alpha=e_1\pm e_3, e_2\pm e_3, e_1+e_2, 2e_1, 2e_2$. 
Then $\bold P$ is the Siegel parabolic subgroup associated to $\Delta-\{2e_3\}$, and 
$\bold Q$ is the parabolic subgroup associated to $\Delta-\{e_2-e_3\}$. Then
$\bold V$ is the Heisenberg group, and the derived group of $\bold L$ is $\bold L'=\bold H\times Spin(9,1)$.

\begin{lemma}\label{parab}
 For $g\in \bold M$ and $p_B\in\bold N$,
$$g p_B=p_{B_1}g,\quad B_1=\nu(g) g B.
$$
\end{lemma}
\begin{proof} By explicit computation, we see that 
$$
g p_B\begin{pmatrix} X\\ \xi\\X'\\ \xi'\end{pmatrix}=\begin{pmatrix} gX+\xi' gB\\ \nu(g) (\xi+(B,X')+(B\times B,X)+\xi' \det(B))\\ g^*X'+2g^*(B\times X)+\xi' g^*(B\times B)\\ \nu(g)^{-1}\xi' \end{pmatrix}.
$$
$$
p_{B_1}g\begin{pmatrix} X\\ \xi\\X'\\ \xi'\end{pmatrix}=\begin{pmatrix} gX+ \nu(g)^{-1}\xi' B_1 \\ \nu(g)\xi+(B_1,g^*X')+(B_1\times B_1,gX)+ \nu(g)^{-1}\xi' \det(B_1)\\
g^*X'+2B_1\times gX+ \nu(g)^{-1}\xi' B_1\times B_1\\ \nu(g)^{-1}\xi' \end{pmatrix}.
$$
By comparing coefficients, we see that $B_1=\nu(g) g B$.
\end{proof}

Denote the element of $V=\bold V(K)$ by 
$$v(x,y,z)=m_{\bar x_1e_{31}}m_{\bar x_2e_{32}}\cdot p_{y_1 e_{13}} p_{y_2 e_{23}}\cdot p_{z},$$ 
where 
\begin{eqnarray*}
&&x=\begin{pmatrix} x_1\\ x_2\end{pmatrix}, \quad y=\begin{pmatrix} y_1\\y_2\end{pmatrix},\quad z=\begin{pmatrix} a& w\\ \bar w&b\end{pmatrix},
\end{eqnarray*}
where $x_1,x_2,y_1,y_2,w\in \frak C_K$ and $a,b\in K$. We identified $z$ with $\begin{pmatrix} z&0\\0&0\end{pmatrix}$ in the definition of $p_z$.
Then by using the above lemma, we can show that
\begin{equation}\label{mult}
v(x,y,z)v(x',y',z')=v(x+x', y+y', z+z'- y ({}^t \bar{x'})-{x'}({}^t \bar y)).
\end{equation}

Now let 
\begin{equation}\label{xyz}
\begin{array}{rl}
X=X(K)=&\{ m_{\bar x_1e_{31}}m_{\bar x_2e_{32}} \in V\ |\ x_1,x_2\in \frak C_K  \},\\
 Y=Y(K)=&\{ p_{y_1 e_{13}} p_{y_2 e_{23}}\in V\ |\ y_1,y_2\in \frak C_K  \},\\
 Z=Z(K)=&\Big\{ p_z\in V \ |\  z\in \frak J_2(K)  \Big\}\simeq \frak J_2(K).
\end{array}
\end{equation}
We identify $X$ (resp. $Y$) with $\frak C^2_K$ by $m_{\bar x_1e_{31}}m_{\bar x_2e_{32}}\mapsto x=\begin{pmatrix} x_1\\ x_2\end{pmatrix}$ 
(resp. by $p_{y_1 e_{13}} p_{y_2 e_{23}}\mapsto y=\begin{pmatrix} y_1\\ y_2\end{pmatrix}$).  
Then we have the decomposition 
\begin{equation}\label{v}
V=\bold V(K)=X\cdot Y\cdot Z.
\end{equation}
We hope that it is clear from the context when $X,Y,Z$ denote the sets, or they denote the elements of $\frak J=\frak J_K$. 

For any $S\in\frak J_2(K)$, define ${\rm tr}_S: Z=\{v(0,0,z)\}\longrightarrow K$, ${\rm tr}_S(v(0,0,z))=\ds\frac{1}{2}(S,z)$. Since $Z$ is the center of $V$, 
${\rm Ker}({\rm tr}_S)$ is a normal subgroup of $V$, and we may consider the quotient $V_0=V/{\rm Ker}({\rm tr}_S)$.

Define the alternating form on $X\oplus Y$ by
$$\langle (x,y), (x',y')\rangle_S={\rm Tr}(S ( x({}^t \bar{y'})+y'({}^t {\bar x})-x' ({}^t \bar{y})-y({}^t {\bar x'}))).
$$
Consider the map $g_S: V\longrightarrow X\oplus Y\oplus K$ by
\begin{equation}\label{alt1}
v(x,y,z)=v(x)v(y)v(z)\longmapsto (x,y, {\rm Tr}(\frac{1}{2}Sz)+{\rm Tr}(\frac S2 (y{}^t {\bar x}+ x {}^t \bar{y}))).
\end{equation}

From (\ref{mult}), we see that $g_S(v(x,y,z)v(x',y',z'))= (x+x',y+y',z'')$ 
where 
$$z''= {\rm Tr}(\frac{1}{2}Sz)+{\rm Tr}(\frac S2 (y{}^t {\bar x}+ x{}^t \bar{y})+{\rm Tr}(\frac{1}{2}Sz')+ 
{\rm Tr}(\frac S2(y'{}^t {\bar x'}+x'{}^t \bar{y'})
+ \frac 12 \langle (x,y), (x',y')\rangle_S.$$
Since ${\rm Ker}(g_S)={\rm Ker}({\rm tr}_S)$, if $\det(S)\not=0$ then we obtain the isomorphism
\begin{equation}\label{heisenberg1}
g_S: V_0=V/{\rm Ker}({\rm tr}_S)\stackrel{\sim}{\lra} X\oplus Y\oplus K.
\end{equation}

Next we compute the action of $\bold H(K)$ on $V=\bold V(K)$: Recall that $\bold H(K)=<p_{be_3}, \iota_{e_3}>\simeq SL_2(K)$ for $b\in K$.
We identify $\gamma=\begin{pmatrix} a&b\\c&d\end{pmatrix}\in SL_2(K)$ with the corresponding element in $\bold H(K)$ under the isomorphism. 
Observe \cite{F} that 
$$e_3\times X=\frac 12 \begin{pmatrix} b&-x&0\\ -\bar x&a&0\\ 0&0&0\end{pmatrix},\quad
e_3\times (e_3\times X)=\frac 14 \begin{pmatrix} a&x&0\\ \bar x&b&0\\ 0&0&0\end{pmatrix}.
$$
Then for $i=1,2$,
\begin{equation}\label{iota}
\iota_{e_3}^{-1}p_{x e_{i3}} \iota_{e_3}=m_{-\bar x e_{3i}},\quad
\iota_{e_3}^{-1}m_{\bar x e_{3i}}\iota_{e_3}=p_{x e_{i3}}.
\end{equation}
For $1\leq i\leq j\leq 2$, $\iota_{e_3}^{-1}p_{x e_{ij}} \iota_{e_3}=p_{x e_{ij}}$.
Hence 
$$p_{be_3}^{-1} v(x,y,z) {p_{be_3}}=v(x,bx+y, z-b x {}^t {\bar x}),\quad \iota_{e_3}^{-1} v(x,y,z) \iota_{e_3}=v(-y,x,z+x {}^t \bar y+y {}^t\bar x).
$$
Since $p_{ce_3}'=\iota_{e_3} p_{-ce_3}\iota_{e_3}^{-1}$, and
$h(a)=\begin{pmatrix} a&0\\0&a^{-1}\end{pmatrix}$ is identified with $p_{ae_3}p_{-a^{-1}e_3}'p_{ae_3} \iota_{e_3}^{-1}$,
we see that
$${p_{ce_3}'}^{-1} v(x,y,z) {p_{ce_3}'}=v(x+cy,y, z-c y {}^t {\bar y}),\quad h(a)^{-1}v(x,y,z)h(a)=v(ax,a^{-1}y,z).
$$
Here $h(a)\in \bold M$, and $\nu(h(a))=a$; More explicitly,
$$h(a)(X,\xi,X',\xi')=(X+(a-1)(e_3,X)e_3, a\xi, X'-(1-a^{-1})(e_3,X')e_3, a^{-1}\xi').
$$
Hence 
\begin{lemma} \label{H}
Let $\gamma=\begin{pmatrix} a&b\\c&d\end{pmatrix}\in \bold H(K)$. Then 
$\gamma^{-1} v(x,y,z) \gamma=v(ax+cy, bx+dy,z')$
where $$z'=z-\frac 12( (ax+cy){}^t\overline{(bx+dy)}+ (bx+dy){}^t\overline{(ax+cy)}-x {}^t\bar y-y {}^t\bar x).$$
\end{lemma}

\section{Jacobi group in $E_{7,3}$, Weil representation, and theta functions}
\subsection{Jacobi group in $E_{7,3}$}
Let $\A$ be the ring of adeles of $\Q$ and $\A_f$ its finite part. 
Let $\widehat{\mathbb{Z}}$ be the profinite completion of $\mathbb{Z}$. 
For $R=\A,\ \A_f,\ \widehat{\Z},\ \Q,\ \Q_p,\ \Z_p$, $p\leq \infty$, or any field $R$, one can consider $X(R)$, $Y(R)$, $Z(R)$, 
$\bold V(R)$ 
(analogues of (\ref{xyz}) and (\ref{v}))) 
by using $\frak C_R$ and the action of $H(R)\simeq SL_2(R)$ on 
$\bold V(R)$ by using the calculation done in Section \ref{E7}. Note that we may not get the identification 
$H(R)\simeq SL_2(R)$ for an arbitrary ring $R$ since this map is described in terms of the root system.  
(The interested readers should consult the notion of Chevalley basis. (cf. \cite{steinberg}))  

Now we introduce a new coordinate on $\bold V$ by modifying the group actions, and define the Jacobi group in $E_{7,3}$:
For any $x,y\in X(R)$ and $S\in \frak  J_2(R)$ so that $\det(S)\not=0$, 
we define 
\begin{equation}\label{quadratic}
\sigma(x,y):=x{}^t\bar{y}+y{}^t\bar{x},\ \sigma_S(x,y):=(S,\sigma(x,y)) \ {\rm and}\ 
\lambda_S(x,y):=\frac{1}{2}\sigma_S(x,y).
\end{equation}
Clearly $\sigma(x,y)\in  \frak  J_2(R)\simeq Z(R)$. A new coordinate $v_1(x,y,z)$ on $\bold V$ is defined by  
$$v_1(x,y,z):=v(x,y,z-\sigma(x,y)),\ x\in X(R),\ y\in Y(R),\ {\rm and}\ z\in Z(R).
$$
Then by (\ref{mult}), one has 
\begin{equation}{\label{mult2}}
v_1(x,y,z)v_1(x',y',z')=v_1(x+x',y+y',z+z'+\sigma(x,y')-\sigma(x',y)).
\end{equation}
The alternating pairing on $X(R)\oplus Y(R)$ is modified as   
\begin{equation}\label{alt}
\langle(x,y),(x',y')\rangle_S:=2(\lambda_S(x,y')-\lambda_S(x',y))=\sigma_S(x,y')-\sigma_S(x',y),
\end{equation}
and $X(R)\oplus Y(R)\oplus R$ has the Heisenberg structure defined by 
$$(x,y,a)\ast(x',y',b)=(x+x',y+y',a+b+\frac{1}{2}\langle(x,y),(x',y')\rangle_S).$$ 
For any  $S\in \frak J_2(\Z)_{+}$, the Heisenberg structure on $V$ is modified by passing to $g_S$ 
(see (\ref{heisenberg1}) for this map) as 
\begin{equation}\label{heisenberg2}
g_{1,S}:\bold V(R)\lra X(R)\oplus Y(R)\oplus R,\ v_1(x,y,z)\mapsto (x,y,\frac{1}{2}(S,z)).
\end{equation}
Noting $\frac{1}{2}(\sigma_S(x,y')-\sigma_S(x',y))=\lambda_S(x,y')-\lambda_S(x',y)$, 
it is easy to see that 
$$g_{1,S}(v_1(x,y,z)v_1(x',y',z'))=g_{1,S}(v_1(x,y,z))\ast g_{1,S}(v_1(x',y',z')),
$$ 
or equivalently, $g_{1,S}$ preserves the Heisenberg structures in both sides. 

The action of $\bold H(R)$ on $\bold V(R)$ with the new coordinates  now turn to be much simpler by Lemma \ref{H}:
\begin{equation}\label{h2}
\gamma^{-1} v_1(x,y,z) \gamma=v_1(ax+cy, bx+dy,z) \ 
{\rm for}\ \gamma=\begin{pmatrix} a&b\\c&d\end{pmatrix}\in \bold H(R). 
\end{equation}

By using this action, we define the Jacobi group in $\bold G(R)$ by 
\begin{equation}\label{jacobi}
J(R):=\bold V(R)\rtimes \bold H(R) 
\end{equation}
with the coordinates $(v_1(x,y,z),h(\gamma)),\ \gamma=\begin{pmatrix} a&b\\c&d\end{pmatrix}$ under the identification 
$\bold H(R)\simeq {\rm SL}_2(R)$. 

\subsection{Weil representation and theta functions}\label{weil-rep}
In this section we shall recall Weil representation and theta series (cf. Section 1-3 of \cite{Ik3}). 

For each place $p\not=\infty$, put 
$$e_p(x)=\exp(-2\pi\sqrt{-1}\cdot{\rm Frac}(x))$$ for $x\in \Q_p$ where 
${\rm Frac}(x)$ stands for the fractional part of $x$. For $p=\infty$, put $\bold{e}(x)=e_\infty(x):=\exp(2\pi \sqrt{-1}x)$ for 
$x\in \R$. 
Fix a non-trivial additive character $\psi:\Q\ba\A\lra \C^\times$ and decompose it as the restricted tensor product 
$\psi=\ds\otimes'_{p\le \infty}\psi_p$. As a standard example, one can take $\psi^{{\rm st}}:=\ds\otimes'_{p\le \infty}e_p(\ast)$ which 
will be used later when we translate the adelic setting into the classical setting and vise versa.  

Fix $S\in \frak J_2(\Z)_{+}$. 
We denote by $h(a)$ (resp. $n(b)$) the element of $\bold H(\Q_p)$ corresponding to 
$\left(\begin{array}{cc}
a & 0 \\
0 & a^{-1}
\end{array}\right),\ a\in \Q^\times_p$ (resp. $\left(\begin{array}{cc}
1 & b  \\
0 & 1
\end{array}\right),\ b\in \Q_p$) under $\bold H(\Q_p)\simeq {\rm SL}_2(\Q_p)$. Note that 
$\left(\begin{array}{cc}
0 & 1 \\
-1 & 0
\end{array}\right)$ corresponds to $\iota_{e_3}\in \bold H(\Q_p)$.  

For each place $p\le \infty$, the Schr\"odinger model $\omega_{S,p}$ on $\bold V(\Q_p)$ 
with the central character $\psi_{p,S}:z\mapsto \psi_p(\frac{1}{2}(S,z)),\ z\in Z(\Q_p)$ realized on the Schwartz space $\mathcal{S}(X(\Q_p))$ is given by 
$$\omega_{S,p}(v_1(x,y,z))\varphi(t)=\varphi(t+x)\psi_p(\frac{1}{2}(S,z)+2\lambda_S(t,y)+\lambda_S(x,y))$$
for $\varphi\in \mathcal{S}(X(\Q_p))$ and $v_1(x,y,z)\in \bold V(\Q_p)$. 
Noting the multiplication law (\ref{mult2}), it is easy to check 
$$\omega_{S,p}(v_1(x,y,z)v_1(x',y',z'))\varphi(t)=\omega_{S,p}(v_1(x,y,z))(\omega_{S,p}(v_1(x',y',z')\varphi(t)).
$$ 
By the Stone-von Neumann theorem, $\omega_{S,p}$ is a unique irreducible unitary representation on which 
$Z(\Q_p)$ acts by $\psi_{p,S}$. 

Recall the conjugate action $\bold H(\Q_p)$ on $\bold V(\Q_p)$ (see (\ref{h2})) and the alternating pairing (\ref{alt}). This induces a homomorphism 
\begin{equation}\label{embed}
\bold H(\Q_p)\hookrightarrow {\rm Sp}_{V/Z}(\Q_p):={\rm Sp}(\bold V(\Q_p)/Z(\Q_p),\langle \ast,\ast\rangle_S).
\end{equation}
Let $\widetilde{{\rm Sp}}_{V/Z}(\Q_p)$ be the metaplectic covering of ${\rm Sp}_{V/Z}(\Q_p)$. This covering does not 
split, but by \cite{kudla-split}, one has  
a splitting  $H(\Q_p)\hookrightarrow  \widetilde{{\rm Sp}}_{V/Z}(\Q_p)$ so that  
the map (\ref{embed}) factors through it via the covering map.  
The Schr\"odinger model $\omega_{S,p}$ extends to the Weil representation of $\bold V(\Q_p)\rtimes \widetilde{{\rm Sp}}_{V/Z}(\Q_p)$ 
acting on $\mathcal{S}(X(\Q_p))$. 
Then the pullback to $\bold H(\Q_p)$ of this representation is given by 
$$\begin{array}{c}
\omega_{S,p}(h(a))\varphi(t)=\chi_S(a)|a|^8_p\varphi(ta),\ \chi_S(a):=\langle \disc(\lambda_S),a \rangle_{\Q_p}\\ 
\omega_{S,p}(n(b))\varphi(t)=\psi_{p}(\lambda_S(t,t)b)\varphi(t) \\
\omega_{S,p}(\iota_{e_3})\varphi(t)=(F_S\varphi)(-t),
\end{array},
$$
where $\langle \ast,\ast \rangle_{\Q_p}$ stands for the Hilbert symbol on $\Q^\times_p\times \Q^\times_p$ and  
$$(F_S\varphi)(t)=\ds\int_{X(\Q_p)} \varphi(x)\psi_p(\lambda_S(t,x))dx,
$$ 
where $dx$ means the Haar measure on $X(\Q_p)$ which is self-dual with respect to the Fourier transform $F_S$. 
Note that the index $8$ of $|a|^8_p$ in the first formula comes from the fact that $\frac{1}{2}{\rm dim}_{\Q_p}X(\Q_p)=8$ and 
we also use $\nu(h(a))=a$. Furthermore, we always have $\chi_S(a)=1$ since $\disc(\lambda_S)$ is a square 
by Lemma \ref{det} in the appendix. 

The global Weil representation of $\omega_S$ of $J(\A)$ acting on the Schwartz space $\mathcal{S}(X(\A))$ is given by the restricted tensor product of $\omega_{S,p}$. 
In our setting, $\omega_S$ is much simpler than that of the case $Sp_{2n}$ (compare to Section 1 of \cite{Ik1}): 
for $\varphi\in \mathcal{S}(X(\A))$,  
$$\begin{array}{c}
\omega_{S}(h(a))\varphi(t)=|a|^8_{\A}\varphi(ta),\quad
\omega_{S}(n(b))\varphi(t)=\psi(\lambda_S(t,t)b)\varphi(t),\quad
\omega_{S}(\iota_{e_3})\varphi(t)=(F_S\varphi)(-t),
\end{array}
$$
where $(F_S\varphi)(t)=\ds\int_{X(\A)} \varphi(x)\psi(\lambda_S(t,x))dx$. 

For each $\varphi\in \mathcal{S}(X(\A))$, the theta function $\Theta^{\psi_S}(v_1(x,y,z)h;\varphi)$ on $\bold V(\A)$ is given by 
$$\begin{array}{rl}
\Theta^{\psi_S}(v_1(x,y,z)h;\varphi)&:=\ds\sum_{\xi\in X(\Q)}\omega_S(v_1(x,y,z)h)\varphi(\xi)\\
&=\ds\sum_{\xi\in X(\Q)}\omega_S(h)\varphi(\xi+x)\psi(\frac{1}{2}(S,z)+2\lambda_S(\xi,y)+\lambda_S(x,y))
\end{array}.
$$
It is easy to see that this function is invariant under the action of $J(\Q)$. 
By the equivalence of the Schr\"odinger model and the lattice model (cf. \cite{takase}), for any $\varphi \in \mathcal{S}(X(\A))$ one has 
the Poisson summation formula which will be used later:
\begin{equation}\label{poisson}
\Theta^{\psi_S}(v_1(x,y,z)(\iota_{e_3} h);\varphi)=\Theta^{\psi_S}(v_1(x,y,z)h;\varphi).
\end{equation}
To end this section, we discuss the relation between the adelic theta function and the classical theta function. 
For any $\varphi' \in \mathcal{S}(X(\A_f))$, we extend this function to an element $\varphi$ of 
$\mathcal{S}(X(\A))$ by

$$\varphi((x_p)_{p\le \infty}):=\varphi_\infty(x_\infty)\varphi'((x_p)_{p< \infty}),\ 
\varphi_\infty(x_\infty)=e^{-2\pi\cdot\sigma_S(x_\infty,x_\infty)}.
$$
Put $\mathfrak{X}:=X(\R)\otimes_\R\C\simeq \C^{\oplus 16}$ and let us extend the quadratic form $\sigma_S$ linearly to that on 
$\mathfrak{X}$. 
For each $\varphi \in \mathcal{S}(X(\A_f))$, 
the classical theta function on $\mathfrak{D}:=\mathbb{H}\times \mathfrak{X}$ is given by 
$$\theta^S_\varphi(\tau,u):=\ds\sum_{\xi\in X(\Q)}\varphi(\xi)\bold{e}(\sigma_S(\xi,\xi)\tau+2\sigma_S(\xi,u)).$$
The group $J(\R)$ acts on $\mathfrak{D}$ by 
$$\beta(\tau,u):=\left(\gamma \tau,\frac{u}{c\tau+d}+x(\gamma \tau)+y \right),
$$
where $\beta=v_1(x,y,z)h$ with $v_1(x,y,z)\in \bold V(\R)$ and $h=h(\gamma)\in \bold H(\R)$ corresponds to 
$\left(\begin{array}{cc}
a & b \\
c & d
\end{array}\right)\in {\rm SL}_2(\R)$. 
Here $\gamma \tau=\ds\frac{a\tau+b}{c\tau+d}$ and put $j(\gamma,\tau):=c\tau+d$ for simplicity. 
For each positive even integer $k$, the automorphy factor on $J(\R)\times \mathfrak{D}$ is defined by 
$$j_{k,S}(\beta,(\tau,u)):=j(\gamma,\tau)^{k}
\bold{e}(-(S,z)+\frac{c}{j(\gamma,\tau)}\sigma_S(u,u)-\frac{2\sigma_S(x,u)}{j(\gamma,u)}-
\sigma_S(x,x)(\gamma\tau)-\sigma_S(x,y)),
$$
for $\beta=v_1(x,y,z)h,\ h=h(\gamma)$ as above. 
After lengthy and painful calculation, one can check the cocycle relation: 
$$j_{k,S}(\beta\beta',(\tau,u))=j_{k,S}(\beta,\beta'(\tau,u))j_{k,S}(\beta',(\tau,u)).$$
For each function $f:\mathfrak{D}\lra \C$ and $\beta\in \bold V(\R)$, 
we define the ``slash" operator $f|_{k,S}[\beta]:\mathfrak{D}\lra \C$ by 
$$f|_{k,S}[\beta](\tau,u):=j_{k,S}(\beta,(\tau,u))^{-1}f(\beta(\tau,u)).$$
Then the following lemma is easy to deduce from the definition:

\begin{lemma}\label{adelic-to-classical} Keep the notation above. 
For each $\varphi \in \mathcal{S}(X(\A_f))$ and $h(\gamma)\in H(\R),\ \gamma\in {\rm SL}_2(\R)$, 

\begin{enumerate}
\item $\Theta^{(\psi^{{\rm st}}_S)^2}(\beta;\varphi')=\theta^S_\varphi|_{8,S}[\beta](\sqrt{-1},0)$ 
for any $\beta \in J(\R)$, 

\item $\theta^S_{\omega_S(\gamma^{-1})\varphi}(\tau,u)=j(\gamma,\tau)^{-8}\theta^S_\varphi(\gamma(\tau,u))$. 
\end{enumerate}
\end{lemma}

\begin{lemma}\label{some-theta}
Keep the notation above. 
Let $\xi$ be an element of $X(\Q)$ so that $\sigma_S(\xi,x)\in \Z$ for all $x\in X(\frak o)$ and $\varphi_\xi$ be 
the characteristic function of $\xi+X(\frak o)$. 
Then $$\Theta^{(\psi^{{\rm st}}_S)^2}(v_1(x,y,z)h;\varphi_\xi)=
\Theta^{(\psi^{{\rm st}}_S)^2}(v_1(x_\infty,y_\infty,z_\infty)h;\varphi_\xi).$$
\end{lemma}
\begin{proof}One can decompose 
any element $v_{1,f}\in J(\A_f)$ as $v_{1,f}=j_{1}\cdot v'_1$ so that $j_{1}\in J(\Q)$ and $v'_{1}=v_1(x',y',z')\in 
\bold V(\widehat{\Z})$. 
Since the sum defining this theta function runs over $X(\Q)\cap (\xi+X(\frak o))$, we see that 
$\psi^{{\rm st}}_S((S,z')+2\sigma_S(a,y')+\sigma_S(x',y'))=1$ for any $a\in X(\Q)\cap (\xi+X(\frak o))$. 
Then the claim follows from the left invariance of the theta function under $J(\Q)$.  
\end{proof}

For any $(\tau,u)\in \mathfrak{D}$, there exist elements $v_1\in \bold V(\R)$ and $g_\infty\in H(\R)$ such that 
$v_1g_\infty (\sqrt{-1},0)=(\tau,u)$ since $1$ and $\tau$ are independent over $\R$. 
From this with Lemma \ref{adelic-to-classical}-(1), we make a bridge between the adelic theta functions and 
the classical theta functions which will be focused in the next section. 

\section{Modular forms on the exceptional domain and Jacobi forms}\label{modular}

We review the definition of modular forms on the exceptional domain $\frak T$ in \cite{B}, and define Jacobi forms for our Jacobi group and study their basic properties. 

\subsection{Modular forms on the exceptional domain}
Let $\Gamma=\bold G(\Z)$ be the arithmetic subgroup of $\bold G(\Q)$ as in \cite{B},
defined by $\bold G(\Z)=\{g\in \bold G(\R) : g\bold W_{\frak o}=\bold W_{\frak o}\}$, where 
$\bold W_{\frak o}=\frak J(\Z)\oplus \Z e\oplus\frak J(\Z)\oplus \Z e'$, and $e=(0,1,0,0)$ and $e'=(0,0,0,1)$. 

\begin{lemma} The arithmetic group $\Gamma$ is generated by $\bold N(\Z)$ and $\iota$.
\end{lemma}
\begin{proof} $\Gamma$ is generated by $\bold N(\Z)$ and $\overline{\bold N}(\Z)$ (\cite{B}, Theorem 5.2), where $\overline{\bold N}$ is the opposite unipotent subgroup of $\bold N$. Since $\overline{\bold N}=\iota^{-1}\bold N\iota$, the result follows.
\end{proof}

\begin{lemma}\label{generator} The arithmetic group $\Gamma$ is generated by $\bold N(\Z)$, $\bold M'(\Z)$ and $\bold H(\Z)$, 
hence by $\bold N(\Z)$, $\bold M'(\Z)$ and $\iota_{e_3}$.
\end{lemma}
\begin{proof} This follows from the above lemma and from the identities, $\iota=\iota_{e_1}\iota_{e_2}\iota_{e_3}$, and
$\iota_{e_2}=\varphi_{23}\iota_{e_3}\varphi_{23}^{-1}$, and $\iota_{e_1}=\varphi_{13}\iota_{e_3}\varphi_{13}^{-1}$, where 
$\varphi_{ij}=m_{e_{ij}}m_{-e_{ji}}m_{e_{ij}}$ for $i\ne j$.
\end{proof}

In \cite{B, Ka, kim}, for $Z\in \frak T$ and $g\in \bold G(\R)$, the action is defined by the right action:
$$Z\cdot g=Z_1,\quad p_Z'g=p_A k p_{Z_1}',\, \text{for $k\in \bold M(\C)$ and $Z_1\in \mathcal H$}.
$$
However, following the usual convention, it is more convenient to define the left action by 
$$gZ=Z_1,\quad g p_Z=p_{Z_1} k p_A',\, \text{for $k\in \bold M(\C)$ and $Z_1\in \frak T$}.
$$
Let $j(g,Z)=\nu(k)^{-1}$ be the canonical factor of automorphy. Then $j(g,Z)$ has the following properties:
$$j(p_B,Z)=1\, \text{for all $B\in\frak J_\R$},\quad j(\iota,Z)=\det(-Z),\quad j(g_1g_2,Z)=j(g_1, g_2 Z)j(g_2,Z).
$$
If $J(Z,g)$ is the functional determinant of $g$ at $Z$, then $J(Z,g)=j(g,Z)^{-18}$. By Lemma \ref{parab}, if $k\in \bold M(\R)$, $kZ$ is just the transformation $\nu(k) (kX+kY\sqrt{-1})$, where $Z=X+Y\sqrt{-1}$, and for $kX$ and $kY$, $k$ is considered as an element of $GL(\frak J)$, and $j(k,Z)=\nu(k)^{-1}$.

\begin{defin} Let $F$ be a holomorphic function on $\frak T$ which for some integer $k>0$ satisfies
$$F(\gamma Z)=F(Z) j(\gamma,Z)^k,\quad Z\in \frak T,\, \gamma\in\Gamma.
$$
Then $F$ is called a modular form on $\frak T$ of weight $k$. We denote by $\mathcal{M}_k(\Gamma)$ the space of such forms. 
For a holomorphic function $F:\frak T\lra \C$, 
the boundary map $\Phi$ is defined by 
$$\Phi F(Z')=\lim_{\tau\to \sqrt{-1}\infty}
F\left(\begin{array}{cc}
Z' & \ast \\
{}^t\bar{\ast} & \tau
\end{array}
\right),
$$
where $Z'\in \frak T_2$.
We call $\mathcal{S}_k(\G):={\rm Ker}(\Phi)\cap \mathcal{M}_k(\Gamma)$ the space of cusp forms of weight $k$ with respect to $\G$. 
We should remark that there is only one equivalent class of cusps since $\bold{G}(\Q)=P(\Q)\bold{G}(\Z)$ (\cite{B}, Theorem 5.2). 
\end{defin}

Since $F(Z+B)=F(Z)$ for $B\in \frak J(\Z)$ and $\frak J(\Z)$ is self-dual, $F$ has a Fourier expansion of the form
\begin{equation}\label{FC}
F(Z)=\sum_{T\in \frak J(\Z)_{\geq 0}} a(T) \bold{e}((T,Z)).
\end{equation}
By Koecher principle, we do not need the holomorphy at the cusps. 

If $F$ is a cusp form, $a(T)=0$ for $T\notin \frak J(\Z)_+$.
  
\subsection{Jacobi forms of matrix index}
We define and study Jacobi forms of matrix index on $\mathfrak{D}=\mathbb{H}\times \mathfrak{X}$ in the classical setting. 
Set 
$$\Gamma_J:=J(\Q)\cap \bold G(\Z).
$$
 
\begin{defin}\label{jacobi} Let $k$ be a positive (even) integer and $S$ be an element of $\frak J_2(\Z)_+$. 
We say a holomorphic function $\phi:\mathfrak{D}\lra \C$ is a Jacobi form (resp. Jacobi cusp form) of weight $k$ and index $S$ if 
$\phi$ satisfies the following conditions:

\begin{enumerate}
\item $\phi|_{k,S}[\beta]=\phi$ for any $\beta\in \Gamma_J$
\item $\phi$ has a Fourier expansion of the form 
$$\phi(\tau,u)=\sum_{\xi\in X(\Q),\ N\in \Z}c(N,\xi)\bold{e}(N\tau+\sigma_S(\xi,u)),
$$
where $c(N,\xi)=0$ unless $S_{\xi,N}:=
\left(\begin{array}{cc}
S & S\xi \\
{}^t \bar{\xi}S & N 
\end{array}\right)$ belongs to $\frak J(\Z)_{\ge 0}$ (resp. $\frak J(\Z)_{+}$).
\end{enumerate}
We denote by $J_{k,S}(\Gamma_J)$ (resp. $J^{{\rm cusp}}_{k,S}(\Gamma_J)$) the space of Jacobi forms 
(resp. Jacobi cusp forms) of weight $k$ and index $S$. 
\end{defin}
Define the dual of the lattice $\Lambda:=X(\Z)=\frak o^2$ with respect to the quadratic form $\sigma_S$ by 
$$\widetilde{\Lambda}(S)=\{x\in X(\Q)\ |\ \sigma_S(x,y)\in \Z {\rm \ for\ all}\ y\in \Lambda \}.$$
If $S\in \frak J_2(\Z)_{+}$, then the quotient $\widetilde{\Lambda}(S)/\Lambda$ is a finite group. 
Fix a complete representative $\Xi(S)$ of $\widetilde{\Lambda}(S)/\Lambda$ and denote by 
$\varphi_\xi$ the characteristic function $\xi+\ds\prod_{p<\infty}X(\o_p)\in \mathcal{S}(X(\A_f))$. 
Any Jacobi form turns to be the sum of products of elliptic modular forms and 
theta functions by following lemma. 
\begin{lemma}\label{product} Assume $S\in \frak J_2(\Z)_{+}$. 
Let $\Xi(S)$ be a complete representative of $\widetilde{\Lambda}(S)/\Lambda$. Then 
any $\phi\in J_{k,S}(\Gamma_J)$ has an expression of the form 
$$\phi(\tau,u)=\sum_{\xi\in \Xi(S)}\phi_{S,\xi}(\tau) \theta^S_{\varphi_\xi}(\tau,u),\quad 
\phi_{S,\xi}(\tau)=\sum_{N\in \Z \atop N-\sigma_S(\xi,\xi)\ge 0}c(N,\xi)\bold{e}((N-\sigma_S(\xi,\xi))\tau).
$$
Furthermore, for each $\xi\in \Xi(S)$,  $\phi_{S,\xi}(\tau)$ is an elliptic modular form of weight $k-8$. 
\end{lemma}
\begin{proof} See example (iv) at Section 2 of \cite{Krieg} and also the argument at p.656 of \cite{Ik1}.  
\end{proof}

Let $k$ be a positive even integer and $F$ be a modular form of weight $k$ on $\mathfrak{T}$. 
Then we have the Fourier-Jacobi expansion 
\begin{equation}\label{FJ}
F\left(\begin{array}{cc}
W & u \\
{}^t\bar{u} & \tau 
\end{array}\right)=\sum_{S\in  \frak J_2(\Z)_{\ge 0}} F_S(\tau,u)\bold{e}((S,W)),\ W\in \frak T_2,\ \tau\in\mathbb{H},\ 
{\rm and}\ u\in X(\R)\otimes_\R\C.
\end{equation}
\begin{lemma}\label{FJC} Keep the notation above. Assume $S\in \frak J(\Z)_+$. 
Then $F_S(\tau,u)\in J_{k,S}(\Gamma_J)$. 
\end{lemma}
\begin{proof}
It is easy to see that $F_S(\tau,u)=\ds\sum_{T\in \frak J_S^+} a(T) \bold{e}(c\tau) \bold{e}((T, {}^t \bar v u)),$
where $\frak J(\Z)_+$ is the set of $T=\begin{pmatrix} S&v\\ {}^t \bar v&c\end{pmatrix}\in \frak J^+$. 
Then the claim follows from the argument at p.656 of \cite{Ik1}. 
\end{proof}

\begin{remark}\label{important-rmk}
Consider any holomorphic function $F(Z),\ Z=\left(\begin{array}{cc}
W & u \\
{}^t\bar{u} & \tau 
\end{array}\right),\ W\in \frak T_2,\ \tau\in\mathbb{H},\ 
{\rm and}\ u\in X(\R)\otimes_\R\C$ on $\frak T$ which is invariant under $\G\cap P(\Q)$. 
Then one has 
the Fourier and Fourier-Jacobi expansion 
$$F(Z)=\sum_{T\in  \frak J(\Z)_{\ge 0}}A_F(T)\bold{e}((T,Z))=\sum_{S\in  \frak J_2(\Z)_{\ge 0}}F_S(\tau,u)\bold{e}((S,W)),
$$ 
as in (\ref{FJ}). 
By the proof of Lemma \ref{product}, 
$$F_S(\tau,u)=\sum_{\xi\in \Xi(S)}F_{S,\xi}(\tau) \theta^S_{\varphi_\xi}(\tau,u),\quad 
F_{S,\xi}(\tau)=\sum_{N\in \Z \atop N-\sigma_S(\xi,\xi)\ge 0}A_F(S_{\xi,N})\bold{e}((N-\sigma_S(\xi,\xi))\tau),
$$
where $S_{\xi,N}=\left(\begin{array}{cc}
S & S\xi \\
{}^t\bar{\xi}S & N 
\end{array}\right)$. In this paper, the function $F_{S,\xi}$ will be called by $(S,\xi)$-component of $F$. 
\end{remark}

We now discuss the relationship between the adelic setting and classical setting. (See \cite{borel&jacquet} for automorphic forms in the adelic setting.) 
Let $\psi$ be a non-trivial additive character of $\Q\ba\A$ and 
for $S\in \frak J_2(\Z)_{+}$, put $\psi_S=\psi\circ {\rm tr}_S:Z\lra \C,\ z\mapsto \psi(\frac{1}{2}(S,z))$. 
\begin{defin}\label{adelic-FJC}
Let $\tilde F$ be an automorphic function on $\bold{G}(\A)$. For each $S\in \frak J_2(\Z)_{+}$, the $S$-th Fourier-Jacobi 
coefficient $F_{\psi_S}$ of $\tilde F$ with respect to $\psi$ is a function on $J(\Q)\ba J(\A)$ given by 
$$F_{\psi_S}(v_1h)=\int_{Z(\Q)\ba Z(\A)} \tilde F(zv_1h)\psi^{-1}_S(z)dz,\ v\in \bold V(\A),\ h\in H(\A).$$
\end{defin}

Let $F$ be a modular form in $\mathcal{M}_{k}(\Gamma)$, and consider its Fourier-Jacobi expansion 
$$
F=\ds\sum_{S\in  \frak J_2(\Z)_{\ge 0}} F_S(\tau,u)\bold{e}((S,W))
$$ 
as above (see (\ref{FJ})). We are using $F_{\psi_S}$ for the Fourier-Jacobi coefficient of $\tilde F$. We hope that this does not cause confusion with $F_S$, which is the Fourier-Jacobi coefficient of $F$.
Let $\tilde F$ denote the automorphic form on $\bold{G}(\A)$ corresponding to $F$ by the strong approximation theorem. Namely, 
\begin{equation}\label{auto-form}
\tilde F(g)=j(g_\infty, E\sqrt{-1})^{-k}F(g_\infty E\sqrt{-1}),\: {\rm for}\  g=\gamma\cdot g_\infty\cdot k'\in \bold{G}(\Q)\bold{G}(\R)K.
\end{equation}
Similarly if we write any element of $J(\A)$ as $v_1h=a\cdot v_{1,\infty}h_\infty\cdot k'_J\in J(\Q)J(\R)K_J$ 
where $K_J=K\cap J(\A)$, 
one has 
$$F_{(\psi^{{\rm st}}_S)^2}(v_1h)=F_S|_{k,S}[v_{1,\infty}h_\infty](\sqrt{-1},0),
$$
by Lemma \ref{FJC}. 
It follows from this that 
$F_{(\psi^{{\rm st}}_S)^2}(v_1h)$ is left invariant under the action of the lattice $\Lambda=X(\frak o)$. 
We also identify $\Lambda$ with a lattice of $Y(\R)$ in an obvious way.

Fix $S\in \frak J_2(\Z)_{+}$. 
For each $\xi \in \Xi(S)$, we put 
$$J^S_{\varphi_\xi}(h;F_{(\psi^{{\rm st}}_S)^2}):=\int_{\bold V(\Q)\ba \bold V(\A)}F_{(\psi^{{\rm st}}_S)^2}(v_1h)
\overline{\Theta^{(\psi^{{\rm st}}_S)^2}(v_1h;\varphi_\xi)}dv_1.$$
Since $F_{(\psi^{{\rm st}}_S)^2}(zv_1h)=(\psi^{{\rm st}}_S)^2(z)F_{(\psi^{{\rm st}}_S)^2}(v_1h)$, one has 
$$J^S_{\varphi_\xi}(h;F_{(\psi^{{\rm st}}_S)^2})=
\int_{(X\oplus Y)(\Q)\ba (X\oplus Y)(\A)}F_{(\psi^{{\rm st}}_S)^2}
(v_1(x,y,0)h)\overline{\Theta^{(\psi^{{\rm st}}_S)^2}(v_1(x,y,0)h;\varphi_\xi)}dv_1(x,y,0).
$$
By Lemma \ref{some-theta}, $\Theta^{(\psi^{{\rm st}}_S)^2}(v_1(x,y,0)h;\varphi_\xi)=
\Theta^{(\psi^{{\rm st}}_S)^2}(v_1(x_\infty,y_\infty,0)h;\varphi_\xi)$.
Then one has 
$$
\begin{array}{rl}
J^S_{\varphi_\xi}(h;F_{(\psi^{{\rm st}}_S)^2})&=
\ds\int_{\Lambda\ba X(\R)\oplus \Lambda\ba Y(\R)}
F_S|_{k,S}[v_{1,\infty}h_\infty](\sqrt{-1},0)\overline{\Theta^{(\psi^{{\rm st}}_S)^2}(v_{1,\infty}h;\varphi_\xi)}dv_{1,\infty}
(x_\infty,y_\infty)
\end{array},
$$
where $v_{1,\infty}(x_\infty,y_\infty)=v_1(x_\infty,y_\infty,0)$. Take 
$h_\infty=\left(
\begin{array}{cc}
y^{\frac{1}{2}}_\infty & x_\infty y^{-\frac{1}{2}}_\infty\\
0 & y^{-\frac{1}{2}}_\infty
\end{array}\right)\in \bold H(\R)$ so that $h_\infty \sqrt{-1}=x_\infty+\sqrt{-1} y_\infty$. Set 
$\tau=h_\infty \sqrt{-1}$ and $v_{1,\infty}h_\infty(\sqrt{-1},0)=(\tau,u)$. Put   
$L_\tau:=\{\lambda_1\tau+\lambda_2\in X(\R)\otimes_\R\C\ |\ \lambda_i\in \Lambda,\ i=1,2 \}$. 
Then by Lemma \ref{adelic-to-classical}, 
one has 
$$
\begin{array}{rl}
& J^S_{\varphi_\xi}(h;F_{(\psi^{{\rm st}}_S)^2})=
\ds\int_{\Lambda\ba X(\R)\oplus \Lambda\ba Y(\R)}
F_S|_{k,S}[v_{1,\infty}h_\infty](\sqrt{-1},0)
\overline{\theta^S_{\varphi_\xi}|_{8,S}[v_{1,\infty}h_\infty](\sqrt{-1},0)}\, dv_{1,\infty}\\
&(\mbox{put $u=x_\infty+\tau y_\infty$})\\
& =\ds\frac{1}{j(g_\infty,i)^{k-8}}
\ds\int_{L_\tau\ba X(\R)\otimes_\R \C}
F_S(\tau,u)\overline{\theta^S_{\varphi_\xi}(\tau,u)}e^{-4\pi({\rm Im}(\tau))^{-1}\sigma_S({\rm Im}(u), {\rm Im}(u))}
\Big|\ds\frac{\partial(x_\infty,y_\infty)}{\partial u}\Big| \, du\\
&=\ds\frac{2^{-8}y^{-8}_\infty}{j(g_\infty,i)^{k-8}}
\ds\int_{L_\tau\ba X(\R)\otimes_\R \C}
F_S(\tau,u)\overline{\theta^S_{\varphi_\xi}(\tau,u)}e^{-4\pi({\rm Im}(\tau))^{-1}\sigma_S({\rm Im}(u), {\rm Im}(u))} \, du\\
&\mbox{(by Lemma \ref{product})} \\ 
&=\ds\frac{2^{-8}y^{-8}_\infty}{j(g_\infty,i)^{k-8}}
\ds\int_{L_\tau\ba X(\R)\otimes_\R \C}
\sum_{\xi'\in \Xi(S)}F_{S,\xi'}(\tau) \theta^S_{\varphi_{\xi'}}(\tau,u)\overline{\theta^S_{\varphi_\xi}(\tau,u)}e^{-4\pi({\rm Im}(\tau))^{-1}\sigma_S({\rm Im}(u), {\rm Im}(u))}du \\
&=\ds\frac{2^{-8}y^{-8}_\infty}{j(g_\infty,i)^{k-8}}F_{S,\xi}(\tau)2^{-24}\det(S)^{-4}y^{8}_\infty \\
&= 2^{-32}\det(S)^{-4}j(h_\infty,i)^{-(k-8)}F_{S,\xi}(\tau). 
\end{array}
$$
Here we used the following formula to get the last equality: for each $\xi$, 
$$\ds\int_{L_\tau\backslash X(\R)\otimes_\R \C}
 \theta^S_{\varphi_{\xi'}}(\tau,u)\overline{\theta^S_{\varphi_\xi}(\tau,u)}e^{-4\pi({\rm Im}(\tau))^{-1}
\sigma_S({\rm Im}(u), {\rm Im}(u))}du
=\left\{
\begin{array}{ll}
2^{-24}\det(S)^{-4}y^{8}_\infty & {\rm if}\ \xi'=\xi \\
0& {\rm otherwise}
\end{array}\right.
$$ 
(Apply Lemma \ref{ortho} for $n=16$ and combine this with ${\rm disc}(\sigma_S)=2^{16}{\rm disc}(\lambda_S)=
2^{16}\det(S)^8$ by Lemma \ref{det}.) 

Summing up, we have proved the following  
\begin{lemma}\label{relation} $j(h_\infty,i)^{(k-8)}J^S_{\varphi_\xi}(h_\infty;F_{(\psi^{{\rm st}}_S)^2})=C_S 
F_{S,\xi}(\tau),\ \tau=h_\infty i$, 
\newline where $C_S=2^{-32}\det(S)^{-4}$. 
\end{lemma}
In the next section, we will prove 
$J^S_{\varphi_\xi}(h;F_{(\psi^{{\rm st}}_S)^2})$ is a section of a degenerate principal series representation of $SL_2(\A)$ if 
$\tilde F$ is an (adelic) Eisenstein series on $\bold{G}(\A)$. By Lemma \ref{relation} above, we will conclude that 
$F_{S,\xi}(\tau)$ is an Eisenstein series in the classical sense. 

\section{Eisenstein series and their Fourier coefficients}

Recall from \cite{kim} an Eisenstein series: Let $\Gamma_{\infty}=\Gamma\cap \bold N(\Bbb Q)$. 
For $l$ a positive integer and $s\in \C$,
$$E_{2l,s}(Z)=\sum_{\gamma\in\Gamma_{\infty}\backslash\Gamma} j(\gamma, Z)^{-2l}|j(\gamma,Z)|^{-s}.
$$

When $s=0$ and $2l>18$, Karel \cite{Ka} computed the Fourier coefficients and showed that they have bounded denominators.
Let 
$$E_{2l}(Z)=E_{2l,0}(Z)=\sum_{T\in \frak J(\Z)_+} a_{2l}(T) \bold{e}((T,Z)).
$$
\begin{theorem}\cite{Ka}\label{fourier1} For $T\in \frak J(\Z)_+$,
$$a_{2l}(T)=C_{2l} \det(T)^{2l-9} \prod_{p| \det(T)} f_T^p(p^{9-2l}),
$$
where $C_{2l}=2^{15}\displaystyle\prod_{n=0}^2 \frac {2l-4n}{B_{2l-4n}}$, and $f_T^p$ is a monic polynomial with rational integer coefficients of degree $d=\text{{\rm ord}}_p (\det(T))$. It satisfies the functional equation
$$X^d f_T^p(X^{-1})=f_T^p(X).
$$
\end{theorem}
Here $B_{2k}$ is the Bernoulli number; $\zeta(2k)=\ds\frac {2^{2k-1}\pi^{2k}B_{2k}}{(2k)!}$. If $n_{2l}$ is the numerator of $\displaystyle\prod_{n=0}^2 B_{2l-4n}$, then $n_{2l} E_{2l}(Z)$ has rational integer Fourier coefficients.
The functional equation of $f_T^p$ is implicit in \cite{Ka}, and it is stated explicitly in \cite{kim}, page 185. 

\begin{corollary}\label{fourier2}
Keep the notation in the theorem above. Set $\widetilde{f}^p_T(X):=X^{d}f_T^p(X^{-2})$, where $d={\rm ord}_p (\det(T))$. 
Then 
$$a_{2l}(T)=C_{2l}\det(T)^{\frac{2l-9}{2}} \prod_{p| \det(T)} \widetilde{f}^p_T(p^{\frac{2l-9}{2}}),
$$
and $\widetilde{f}^p_T(X)=\widetilde{f}^p_T(X^{-1})$. 
\end{corollary}

We can interpret this from the degenerate principal series as in the Siegel case \cite{Ku}. 
Let $K_{\infty}$ be the stabilizer of $E\sqrt{-1}$ in $\bold G(\R)$, where $E={\rm diag}(1,1,1)\in \frak J(\Q)$. It is a maximal compact subgroup of $\bold G(\R)$, 
and its complexification $K_{\infty,\C}$ is conjugate in $\bold G(\C)$ to $\bold M(\C)$ by the Cayley transform.
Let $K=K_{\infty}\prod_p K_p$, where $K_p=\bold G(\Z_p)$. 
By the strong approximation theorem, $\bold G(\Bbb A)=\bold G(\Q)\bold G(\R) K$.

For $s\in\C$, let $I(s)$ be the degenerate principal series representation of $\bold{G}(\A)$ consisting of any smooth, $K$-finite   
function $f:\bold{G}(\A)\lra \C$ such that 
$$f(pg)=\delta^{\frac{1}{2}}_{\bold P}(p)|\nu(p)|^s_{\A}(g)$$
 for any  $p\in \bold P(\A)$ and any  $g\in \bold{G}(\A)$ 
where $\bold P=\bold M\bold N$ is the Siegel parabolic subgroup.
Note that the modulus character $\delta_P$ is given by $\delta_{\bold P}(mn)=|\nu(m)|^{18}_{\A}$. We denote it also by $I(s)=\mbox{Ind}_{\bold P(\Q_p)}^{\bold G(\Q_p)} \ |\nu(g)|^s$.

Let $\Phi(g, s)=\Phi_\infty(g, s)\otimes \otimes_p \Phi_p(g, s)$ be a standard section in $I(s)$. Then one can define the Siegel Eisenstein series
$$E(g,s,\Phi)=\sum_{\gamma\in \bold P(\Q)\backslash \bold{G}(\Q)} \, \Phi(\gamma g, s).
$$
It satisfies the functional equation
$$E(g,s,\Phi)=E(g,-s,M(s)\Phi),\quad M(s): I(s)\longrightarrow I(-s),\quad M(s)\Phi(g)=\int_{\bold N(\A)} \Phi(ng,s)\, dn.
$$
Now $\bold G(\R)=\bold P(\R)K_{\infty}$, and hence $\Phi_\infty$ is determined by its restriction to $K_\infty$. We choose
$$\Phi_\infty(k, s)=\nu(\bold k)^{2l},
$$
where $\bold k\in \bold M(\C)$ corresponds to $k\in K_\infty$ by the Cayley transform. 
Hence $\Phi(mnk,s)=|\nu(m)|^{s+9}_{\A}\nu(\bold k)^{2l}$.

By \cite{B}, page 527, given $Y\sqrt{-1} \in\frak T$, there exists $m\in \bold M(\R)$ such that $m (E\sqrt{-1})=Y\sqrt{-1}$. 
Hence $p_{X}m (E\sqrt{-1})=X+ Y\sqrt{-1}$. Let $g=p_{X}m$. 

Now for $\gamma\in\Gamma$, by Iwasawa decomposition, $\gamma g=n m' k$ with $n\in\bold N(\R)$, $m'\in\bold M(\R)$, and $k\in K_\infty$.
Then 
$$\gamma g (E\sqrt{-1})=\gamma Z=n m' (E\sqrt{-1})=X_1+ Y_1\sqrt{-1}.
$$
Hence $m'(E\sqrt{-1})=Y_1\sqrt{-1}$ and $n=p_{X_1}$.
On the other hand, 
$$j(\gamma g, E\sqrt{-1})=j(\gamma,Z)j(g, E\sqrt{-1})=j(m', E\sqrt{-1})j(k,E\sqrt{-1}).
$$
Here $j(g,E\sqrt{-1})=j(m,E\sqrt{-1})=\det(Y)^{-1}$, $j(m',E\sqrt{-1})=\det(Y_1)^{-1}$. By \cite{BB}, page 500,

\begin{lemma} For $k\in K_\infty$, $j(k,E\sqrt{-1})=\nu(\bold k)^{-1}$, and hence $|j(k,E\sqrt{-1})|=1$.
\end{lemma}

So
$$
\det(Y_1)=\frac {\det(Y)}{|j(\gamma,Z)|},\quad j(k,E\sqrt{-1})=\frac {j(\gamma,Z)}{|j(\gamma,Z)|}.
$$

Therefore,
$$\Phi_\infty(\gamma g, s)=\nu(m')^{-s-9} \nu(\bold k)^{2l}=\det(Y)^{s+9}j(\gamma,Z)^{-2l}|j(\gamma,Z)|^{-s-9+2l}.
$$

Hence as in \cite{Ku}, for $\Phi(g, s)=\Phi_\infty(g, s)\otimes \otimes_p \Phi_p(g, s)$, $\Phi_\infty(g, s)=\nu(\bold k)^{2l}$, and $\Phi_p(g,s)=\Phi_p^0(g,s)$, the normalized spherical section for all $p$,

$$E(g,s,\Phi)=\det(Y)^{s+9} \sum_{\gamma\in \Gamma_\infty\backslash \Gamma} j(\gamma,Z)^{-2l}|j(\gamma,Z)|^{-s-9+2l}.
$$
Hence
$$E(g,s,\Phi)=\det(Y)^{s+9} E_{2l,s+9-2l}(Z)=j(g,E\sqrt{-1})^{-(s+9)}E_{2l,s+9-2l}(Z).
$$
Summing up, we have proved the following:
\begin{prop}\label{adelic-Eisen} The adelic Eisenstein series $E(g,2l-9,\Phi)$ on $\bold{G}(\A)$ which is associated to a standard section of  
$I(2l-9)$ corresponds to $E_{2l,0}(Z)$ via (\ref{auto-form}). 
\end{prop} 

Let $I(s)=\otimes I_p(s)$ and $I_p(s)$ be the $p$-adic degenerate principal series. Then we have

\begin{prop}\label{W}$($\cite{W}$)$ $I_p(s)$ is irreducible except at $s=\pm 1, \pm 5, \pm 9$.
\end{prop}

\begin{remark} In terms of representation theory, the singular modular forms of weight 4 and 8 constructed in \cite{kim} are subrepresentations of $I(s)$ when $s=-5, -1$, resp.
\end{remark}

\section{Fourier-Jacobi expansion of Eisenstein series on $E_{7,3}$}
As seen in Section \ref{weil-rep} (see Lemma \ref{product} and Lemma \ref{FJC}), 
for each $S\in \frak J(\Z)_{+}$, the $S$-th Fourier-Jacobi coefficient of 
a modular form $F$ on $\frak T$ is represented by the sum of the products of theta series and elliptic modular forms. 
In this section we shall prove these elliptic modular forms turn to be Eisenstein series on $\mathbb{H}$ if 
$F$ is an Eisenstein series. To do this we generalize the argument in Section 3 of \cite{Ik3} in our setting and 
by virtue of Lemma \ref{relation} this enable us to work on the adelic setting which is much simpler than the classical setting.  

Let $\omega$ be a unitary character of $\Q^\times\backslash\A^\times$ and $s\in \C$. 
Let $\mathbb{K}=SL_2(\widehat{\Z})\times SO(2)$ be the standard maximal compact subgroup of $SL_2(\A)$. 
We denote by $I(\omega,s)$, the degenerate principal series representation of $\bold{G}(\A)$ consisting of any  
function $f:\bold{G}(\A)\lra \C$ such that 
$$f(pg)=\delta^{\frac{1}{2}}_{\bold P}(p)|\nu(p)|^s_\A\omega(\nu(p))f(g)$$
 for any  $p\in \bold P(\A)$ and any  $g\in \bold{G}(\A)$. Recall that $\delta^{\frac{1}{2}}_P(mn)=|\nu(m)|^9_\A$. 
Similarly we also define  the space $I_1(\omega,s)$  consisting of any smooth, $\mathbb{K}$-finite 
function $f:{\rm SL}_2(\A)\lra \C$ such that 
$$f(pg)=\delta^{\frac{1}{2}}_B(p)|a|^s_{\A}\omega(\nu(p))f(g)$$
 for any  $p=
\left(\begin{array}{cc}
a & b \\
0 & a^{-1} 
\end{array}
\right)\in B(\A)$ and any  $g\in {\rm SL}_2(\A)$. 
Here $B$ is the Borel subgroup of ${\rm SL}_2$ which consists of upper-triangular matrices and 
$\delta^{\frac{1}{2}}_B(p)=|a|_\A$ for $p=
\left(\begin{array}{cc}
a & b \\
0 & a^{-1} 
\end{array}
\right)\in B(\A)
$.    
For any section $f\in I(\omega,s)$, we define the Eisenstein series on $\bold{G}(\A)$ of type $(\omega,s)$ by 
$$E(g;f):=\sum_{\bold P(\Q)\ba \bold{G}(\Q)}f(\gamma g),\ g\in \bold{G}(\A).$$
Let $\psi$ be a non-trivial additive character of $\Q\ba\A$ and 
for $S\in \frak J_2(\Z)_{+}$, put $\psi_S=\psi\circ {\rm tr}_S:Z(\A)\lra \C$. 
Consider the $S$-th Fourier-Jacobi coefficient $E_S(v_1h;f)$ of $E_S(g;f)$ with respect to $\psi_S$ (see Definition \ref{adelic-FJC}). 
For each $\varphi\in \mathcal{S}(X(\A))$, put 
\begin{equation}\label{Fourier-Jacobi}
E_{\psi_S,\varphi}(h):=\int_{\bold V(\Q)\backslash \bold V(\A)} E_S(v_1h,f)\overline{\Theta^{\psi_S}(v_1h;\varphi)}\, dv_1,\ h\in {\rm SL}_2(\A).
\end{equation}
The main purpose in this section is to prove the following key theorem:
\begin{theorem}\label{key-thm}Keep the notation above. Assume that $\varphi$ is $\mathbb{K}$-finite, hence 
the $\C$-span $\langle  \omega_S(k)\varphi\ |\ k\in \mathbb{K} \rangle_\C$ is finite-dimensional. 
For ${\rm Re}(s)\gg 0$,   
\begin{enumerate}
\item $R(h;f,\varphi):=\ds\int_{\bold V(\A)}f(\iota\cdot v_1\cdot \iota_{e_3}\cdot h)\overline{\omega_S(v_1(\iota_{e_3}\cdot h))
\varphi(0)}dv_1$ is a section of $I(\omega,s)$, 
\item $E_{\psi_S,\varphi}$ is an Eisenstein series on $SL_2(\A)$ associated to $R(h;f,\varphi)$.
\end{enumerate}
\end{theorem}
To prove this, we need some lemmas: Let $P=\bold P(\Q), G=\bold G(\Q), Q=\bold Q(\Q)$. Note that $Q$ is the normalizer of $\bold V(\Q)$ in $G$.
The double coset $P\backslash G/Q$ is bijective to the double coset of the Weyl group $W_P\backslash W_G/W_Q$. By \cite{Ca}, page 64,
each double coset of $W_P\backslash W_G/W_Q$ has unique element of minimal length, and they are
$\{1, c_3 (2 3), c_2c_3 (1 3)\}$, where $c_i$ is the Weyl group element attached to $2e_i$, and $(i j)$ is the Weyl group element attached to $e_i-e_j$. Then $G=P\xi_2 Q\cup P\xi_1 Q\cup P\xi_0 Q$, and $P\xi_0 Q$ is the unique open cell, where $\xi_2=1$, $\xi_1=c_3 (2 3)$, and $\xi_0=c_2c_3 (1 3)$. 
In terms of the notation in \cite{B}, page 517, 
$\xi_2=1$, $\xi_1=\iota_{e_3}\varphi_{23}$, and $\xi_0=\iota_{e_2}\iota_{e_3}\varphi_{13}$, where 
$\varphi_{ij}=m_{e_{ij}}m_{-e_{ji}}m_{e_{ij}}$ for $i\ne j$.

\begin{lemma}\label{lem1} For any $q\in Q$, $q$ normalizes $Z(\A)$, and if $\gamma\in G$ is not contained in the open cell $P\xi_0 Q$, then $\psi_S$ is non-trivial on $\gamma^{-1}\bold P(\A)\gamma\cap Z(\A)$.
\end{lemma}
\begin{proof} Let $q=lv$ for $l\in \bold{L}(\Q)$ and $v\in \bold V(\Q)$, and $p_z\in Z(\A)$. Then since $Z$ is the center of $V$, 
$qp_zq^{-1}=(lv)p_z(lv)^{-1}=lp_zl^{-1}$.
If $l$ is in the central torus of $\bold{L}$, $lp_zl^{-1}=p_z$. Otherwise, $l\in \bold L'(\Q)$. Here $\bold L'=\bold H\times Spin(9,1)$, 
where $Spin(9,1)$ is spanned by the unipotent subgroups $m_{x e_{12}}$ and $m_{x e_{21}}$. 
If $l\in \bold H(\Q)$, by Lemma \ref{H}, $lp_zl^{-1}=p_z$.
Suppose $l=m_{xe_{12}}$. Then by Lemma \ref{parab},
$m_{xe_{12}}p_z=p_B m_{xe_{12}}$ for $B=m_{xe_{12}}z=p_{z'}\in Z(\A)$. Similarly, $m_{xe_{21}}p_z=p_{z"}\in Z(\A)$.
Hence we have proved $qZ(\A)q^{-1}=Z(\A)$.

We may assume that $\gamma=\xi_1, \xi_2$. If $\gamma=\xi_2=1$, $P(\A)\cap Z(\A)=Z(\A)$. So $\psi_S$ is not trivial on $\bold P(\A)\cap Z(\A)$.
Let $\gamma=\xi_1$. Using (\ref{iota}), we can compute easily that $\gamma^{-1} m_{\bar xe_{31}}\gamma= p_{xe_{12}}\in Z(\A)$. Hence $\gamma^{-1}\bold P(\A)\gamma\cap Z(\A)$ contains the subgroup $\{p_{xe_{12}}\ |\ x\in \frak C_{\A}\}$. So $\psi_S$ is not trivial on $\gamma^{-1}\bold P(\A)\gamma\cap Z(\A)$.
\end{proof}
Let $P_H$ be the Borel subgroup of $H$ consisting of upper triangular matrices. 
\begin{lemma}\label{lem2} The right coset can be written as 
$P\backslash P\xi_0 Q=\xi_0 \cdot (Y(\Q)\backslash \bold V(\Q))\cdot (P_H(\Q)\backslash H(\Q))$. 
\end{lemma}
\begin{proof} We can write $q\in Q$ as $q=slvh$ with $s$ in the central torus, $l\in Spin(9,1)(\Q)$, $v\in \bold V(\Q)$, and 
$h\in H(\Q)$.
It is easy to show that $\xi_0 l\xi_0^{-1}\in \bold{M}'(\Q)$, and $\xi_0 v(y) \xi_0^{-1}\in  \bold{M}'(\Q)$, and 
$\xi_0 p_{ae_3} \xi_0^{-1}=p_{ae_1}$;
By direct computation, $\xi_0 m_{x e_{21}}\xi_0^{-1}=m_{x e_{32}}$, and $\xi_0 m_{x e_{12}}\xi_0^{-1}=m_{x e_{23}}$.
And $\xi_0 p_{ye_{13}}\xi_0^{-1}=m_{\bar y e_{31}}$, and $\xi_0 p_{ye_{23}}\xi_0^{-1}=m_{\bar y e_{21}}$.
Note that $h(a)$ is identified with $p_{ae_3}p_{-a^{-1}e_3}'p_{ae_3}\iota_{e_3}^{-1}$. Hence 
$\xi_0 h(a)\xi_0^{-1}=p_{ae_1}p_{-a^{-1}e_1}'p_{ae_1}\iota_{e_1}^{-1}\in  \bold{M}'(\Q)$ giving the claim.
\end{proof}
We have in analogy to \cite{Ik3}, p 630:
\begin{lemma} \label{Ik-lem} 
\begin{enumerate} 
\item $\iota \cdot v_1(0,y,z) \iota_{e_3} p_{be_3}=p_{be_3} k\cdot \iota \cdot v_1(0,y,z+b y {}^t \bar{y}) \iota_{e_3}$ 
where $k=m_{by_1 e_{13}}m_{by_2 e_{23}}$ with $\nu(k)=1$.
\item $\iota\cdot v_1(0,y,z) \iota_{e_3} h(a)=h(a)\cdot \iota \cdot v_1(0,ay,z) \iota_{e_3}$
with $\nu(h(a))=a$. 
\item $\varphi_{13}\xi_0 \iota_{e_3}=\iota$ with $\nu(\varphi_{13})=1$ and $\iota^2_{e_3}=-1$.
\end{enumerate}
\end{lemma}
\begin{proof} Note that $v(0,y,z)=v_1(0,y,z)$; (2) is straightforward by using (1), and
$\iota_{e_i} h(a)\iota_{e_i}^{-1}=h(a)$ for $i=1,2$, and $\iota_{e_3} h(a)\iota_{e_3}^{-1}=h(a^{-1})$; For (1), use $\iota_{e_i} p_{be_3}\iota_{e_i}^{-1}=p_{be_3}$ for $i=1,2$, and
$\iota\cdot m_{\bar x e_{3i}} \cdot \iota^{-1}=m_{-x e_{i3}}$ for $i=1,2$; (3) follows from the fact that $\varphi_{13}\in\bold M'$, and $\iota_{e_1}=\varphi_{13}\iota_{e_3}\varphi_{13}$.
\end{proof} 

\noindent{\it Proof of Theorem \ref{key-thm}}. We first prove (2), namely,
$$E_{\psi_S,\varphi}(h)=\sum_{\gamma\in P_H(\Q)\ba H(\Q)}R(\gamma h;f,\varphi).
$$
This series will be convergent for Re$(s)\gg 0$ provided if the first assertion holds (\cite{langlands}). 
In fact, one has 
$$
\begin{array}{ll}
&E_{\psi_S,\varphi}(h)=\ds\int_{\bold V(\Q)\backslash \bold V(\A)} E_S(v_1h,f)\overline{\Theta^{\psi_S}(v_1h;\varphi)}\, dv_1 =\ds\int_{\bold V(\Q)\backslash \bold V(\A)} E(v_1h,f)\overline{\Theta^{\psi_S}(v_1h;\varphi)}\, dv_1 \\
&=\ds\sum_{i=1,2}\sum_{\gamma\in P\ba P\xi_i Q}\ds\int_{\bold V(\Q)\backslash \bold V(\A)}f(\gamma v_1h)
\overline{\Theta^{\psi_S}(v_1h;\varphi)}\, dv_1 +\ds\sum_{\gamma\in P\ba P\xi_0 Q}\ds\int_{\bold V(\Q)\backslash \bold V(\A)}f(\gamma v_1h)
\overline{\Theta^{\psi_S}(v_1h;\varphi)}\, dv_1. 
\end{array}
$$
In the first integral above, by Lemma \ref{lem1}, there exists an element 
$z_0=\gamma^{-1}p\gamma\in Z(\A)\cap \gamma^{-1}P(\A)\gamma$ such that $\psi_S(z_0)\not=1$. Clearly $\nu(p)=1$.  
Then one has 
$$
\begin{array}{ll}
&\ds\int_{\bold V(\Q)\backslash \bold V(\A)}f(\gamma v_1h)
\overline{\Theta^{\psi_S}(v_1h;\varphi)}\, dv_1 = \ds\int_{\bold V(\Q)\backslash \bold V(\A)}f(\gamma z_0v_1h)
\overline{\Theta^{\psi_S}(z_0v_1h;\varphi)}\, d(z_0v_1)\\
&=\overline{\psi_S(z_0)}\ds\int_{\bold V(\Q)\backslash \bold V(\A)}f(p\gamma v_1h)
\overline{\Theta^{\psi_S}(v_1h;\varphi)}\, dv_1 = \overline{\psi_S(z_0)}\ds\int_{\bold V(\Q)\backslash \bold V(\A)}f(\gamma v_1h)
\overline{\Theta^{\psi_S}(v_1h;\varphi)}\, dv_1,
\end{array}
$$
which claims the vanishing of $\ds\int_{\bold V(\Q)\backslash \bold V(\A)}f(\gamma v_1h)
\overline{\Theta^{\psi_S}(v_1h;\varphi)}\, dv_1$. 
By Lemma \ref{lem2}, 
$$
\begin{array}{ll}
E_{\psi_S,\varphi}(h)&=\ds\sum_{\gamma_1\in Y(\Q)\ba \bold V(\Q)}
\sum_{\gamma\in P_H(\Q)\ba H(\Q)}\ds\int_{\bold V(\Q)\backslash \bold V(\A)}f(\xi_0\gamma_1 \gamma v_1h)
\overline{\Theta^{\psi_S}(v_1h;\varphi)}\, dv_1 \\
&=\ds\sum_{\gamma_1\in Y(\Q)\ba \bold V(\Q)}
\sum_{\gamma\in P_H(\Q)\ba H(\Q)}\ds\int_{\bold V(\Q)\backslash \bold V(\A)}f(\xi_0\gamma_1 \gamma v_1h)
\overline{\Theta^{\psi_S}(v_1h;\varphi)}\, dv_1 \\
&({\rm transform }\ v_1\ {\rm into}\ \gamma^{-1} v_1 \gamma)\\
&=\ds\sum_{\gamma_1\in Y(\Q)\ba \bold V(\Q)}
\sum_{\gamma\in P_H(\Q)\ba H(\Q)}\ds\int_{\bold V(\Q)\backslash \bold V(\A)}f(\xi_0\gamma_1 v_1\gamma h)
\overline{\Theta^{\psi_S}((\gamma^{-1} v_1 \gamma) h;\varphi)}\, dv_1 \\
& ( {\rm use }\ J(\Q){\rm -invariance\ of}\  \Theta^{\psi_S})\\
&=\ds\sum_{\gamma_1\in Y(\Q)\ba \bold V(\Q)}
\sum_{\gamma\in P_H(\Q)\ba H(\Q)}\ds\int_{\bold V(\Q)\backslash \bold V(\A)}f(\xi_0\gamma_1 v_1\gamma h)
\overline{\Theta^{\psi_S}(\gamma_1 v_1 (\gamma h);\varphi)}\, dv_1 \\
&=\ds\sum_{\gamma\in P_H(\Q)\ba H(\Q)}\ds\int_{Y(\Q)\backslash \bold V(\A)}f(\xi_0 v_1\gamma h)
\overline{\Theta^{\psi_S}(v_1 (\gamma h);\varphi)}\, dv_1 \\
&({\rm Poisson\ summation\ formula\ (\ref{poisson})}) \\
&=\ds\sum_{\gamma\in P_H(\Q)\ba H(\Q)}\ds\int_{Y(\Q)\backslash \bold V(\A)}f(\xi_0 v_1\gamma h)
\sum_{\ell\in Y(\Q)}\overline{F_S(\omega_S((-\ell)v_1\gamma h)\varphi(0))}\, dv_1. 
\end{array}
$$
Transforming $v_1$ into $(-\ell)^{-1}v_1$, one has $f(\xi_0(-\ell)^{-1}v_1\gamma h)=f(\xi_0\gamma h)$, since 
$\xi_0$ and $\ell$ (hence $(-\ell)^{-1}$) are commutative up to the multiplication by an element of $P(\Q)$, and $|\nu(P(\Q))|_{\A}=1$.   
Hence 
$$
\begin{array}{ll}
E_{\psi_S,\varphi}(h)&=\ds\sum_{\gamma\in P_H(\Q)\ba H(\Q)}\ds\int_{Y(\Q)\backslash \bold V(\A)}f(\xi_0 v_1\gamma h)
\sum_{\ell\in Y(\Q)}\overline{F_S(\omega_S(v_1\gamma h)\varphi(0))}\, dv_1 \\
&=\ds\sum_{\gamma\in P_H(\Q)\ba H(\Q)}\ds\int_{ \bold V(\A)}f(\xi_0 v_1\gamma h)
\overline{F_S(\omega_S(v_1\gamma h)\varphi(0))}\, dv_1 \\
&=\ds\sum_{\gamma\in P_H(\Q)\ba H(\Q)}\ds\int_{ \bold V(\A)}f(\xi_0 v_1(x,y,z)\gamma h)
\overline{\omega_S(\iota_{e_3}v_1(-x,y,z)\gamma h)\varphi(0)}\, dv_1(x,y,z). 
\end{array}
$$
It is easy to see that $\xi_0$ commutes with $X(\A)$ (see the proof of Lemma \ref{lem2}), $v_1(2x,0,0)v_1(-x,y,z)=v_1(x,y,z)$, and $\nu(X(\A))=1$. Hence 
$$
\begin{array}{ll}
E_{\psi_S,\varphi}(h)&=\ds\sum_{\gamma\in P_H(\Q)\ba H(\Q)}\ds\int_{ \bold V(\A)}f(\xi_0 v_1\gamma h)
\overline{\omega_S(\iota_{e_3}v_1\gamma h)\varphi(0)}\, dv_1 \\
&({\rm transform }\ v_1\ {\rm into}\ {\iota_{e_3}}^{-1} v_1 \iota_{e_3})\\
&=\ds\sum_{\gamma\in P_H(\Q)\ba H(\Q)}\ds\int_{ \bold V(\A)}f(\xi_0 \iota^{-1}_{e_3} v_1 \iota_{e_3}\gamma h)
\overline{\omega_S(v_1\iota_{e_3}\gamma h)\varphi(0)}\, dv_1.
\end{array}
$$
By Lemma \ref{Ik-lem} (3),
$$
E_{\psi_S,\varphi}(h)=\ds\sum_{\gamma\in P_H(\Q)\ba H(\Q)}\ds\int_{ \bold V(\A)}f(\iota v_1 \iota_{e_3}\gamma h)
\overline{\omega_S(v_1\iota_{e_3}\gamma h)\varphi(0)}\, dv_1
=\ds\sum_{\gamma\in P_H(\Q)\ba H(\Q)}R(\gamma h;f,\varphi). 
$$

We now prove (1). Noting that 
$$\iota\cdot v_1(x,y,z)=\iota\cdot v_1(x,0,0)\iota^{-1}\cdot \iota\cdot v_1(0,y,z)=v_1(0,-\bar{x},0)\cdot \iota\cdot v_1(0,y,z),
$$
and $\nu(v_1(0,-\bar{x},0))=1$, one has 
$$
\begin{array}{ll}
& R( h;f,\varphi)\\
&=\ds\int_{ \bold V(\A)}f(\iota\cdot v_1(x,y,z) \iota_{e_3} h)
\overline{\omega_S(\iota_{e_3} h)\varphi(x)\psi_S(\frac{1}{2}(S,z)+\lambda_S(x,y))}\, dv_1 \\
&=\ds\int_{ X(\A)}\int_{ Y(\A)}\int_{ Z(\A)}f(\iota\cdot v_1(0,y,z) \iota_{e_3} h)
\overline{\omega_S(\iota_{e_3} h)\varphi(x)\psi_S(\frac{1}{2}(S,z)+\lambda_S(x,y))}\, dv_1\\
 &=\ds\int_{ Y(\A)}\int_{ Z(\A)}f(\iota\cdot v_1(0,y,z) \iota_{e_3} h)
\overline{\Big(\int_{ X(\A)} 
\omega_S(\iota_{e_3} h)\varphi(x)\psi_S(\lambda_S(x,y))dx\Big)\psi_S(\frac{1}{2}(S,z))}\, dydz\\
 &=\ds\int_{ Y(\A)}\int_{ Z(\A)}f(\iota\cdot v_1(0,y,z) \iota_{e_3} h)
\overline{\Big(
F_S(\omega_S(\iota_{e_3} h)\varphi)(-y)\Big)\psi_S(\frac{1}{2}(S,z))}\, dydz\\
&=\ds\int_{ Y(\A)}\int_{ Z(\A)}f(\iota\cdot v_1(0,y,z) \iota_{e_3} h)
\overline{(\omega_S(h)\varphi)(y)\psi_S(\frac{1}{2}(S,z))}\, dydz.\\
\end{array}
$$
We now compute the actions of $p_{be_3},\ b\in \A$ and $h(a),\ a\in \A^\times$ respectively. 
By Lemma \ref{Ik-lem}, one has 
$$
\begin{array}{ll}
R(p_{be_3} h;f,\varphi) &=\ds\int_{ Y(\A)}\int_{ Z(\A)}f(\iota\cdot v_1(0,y,z+by{}^t\bar{y}) \iota_{e_3} h)
\overline{(\omega_S(p_{be_3}h)\varphi)(y)\psi_S(\frac{1}{2}(S,z+by{}^t\bar{y}))}\, dydz \\
&({\rm transform }\ z\ {\rm into}\ z+by{}^t\bar{y}\ {\rm and}\ \omega_S(p_{be_3}h)=\omega_S(h))\\
&=R(h;f,\varphi). 
\end{array}
$$
By Lemma \ref{Ik-lem} again, one has 
$$
\begin{array}{ll}
& R(h(a) h;f,\varphi)=\ds\int_{ Y(\A)}\int_{ Z(\A)}f(h(a)\cdot\iota\cdot v_1(0,ay,z) \iota_{e_3} h)
\overline{(\omega_S(h(a) h)\varphi)(y)\psi_S(\frac{1}{2}(S,z))}\, dydz \\
&=\delta^{\frac{1}{2}}_{\bold P}(h(a))|a|^s_{\A} \omega(a)|a|^8_{\A} 
\ds\int_{ Y(\A)}\int_{ Z(\A)}f(\iota\cdot v_1(0,ay,z) \iota_{e_3} h)
\overline{(\omega_S(h)\varphi)(ay)\psi_S(\frac{1}{2}(S,z))}\, dydz.
\end{array}
$$
Transform $y$ into $\frac{y}{a}$ and note that $d(\frac{y}{a})=|a|^{-16}_{\A}dy$ and 
$\delta^{\frac{1}{2}}_{\bold P}(h(a))=|a|^9_{\A}$. So 
$$
R(h(a) h;f,\varphi)=|a|^{1+s}_{\A}\omega(a)R(h;f,\varphi)=\delta^{\frac{1}{2}}_{P_H}(a)|a|^{s}_{\A}\omega(a)R(h;f,\varphi).
$$
The smoothness and $\mathbb{K}$-finiteness follow from those of $f$ and $\varphi$. 
Hence $R(h;f,\varphi)\in I_1(\omega,s)$. $\square$

\section{Compatible family of Eisenstein series}

\begin{defin}
Let $k$ be a positive integer. Let $h(\tau)$ be an elliptic modular form of weight $k$ with respect to $SL_2(\Z)$. 
We denote by $\mathcal V(h)$, the $\C$-vector space spanned by $\{ h|_k[\gamma],\ \gamma\in GL_2(\Q)^+\}$ where 
$h|_k[\gamma](\tau):=j(\gamma,\tau)^{-k}h(\gamma\tau)$. 
\end{defin}

Let $\Phi(X)=\Phi(\{X_p\}_p)=\ds\otimes'_p \Phi_p(X_p)\in \ds\otimes'_p \C[X_p,X_p^{-1}]$ where $p$ runs over all prime numbers.  
Denote by $\mathcal R$ the set of all $\Phi(X)=\Phi(\{X_p\}_p)$ such that $\Phi_p(X_p)=\Phi_p(X_p^{-1})$ for any prime $p$.
For each non-zero sequence of complex numbers $\{a_p\}_p$ indexed by all primes, the value of $\Phi(X)$ at $\{X_p\}_p=\{a_p\}_p$ is denoted by $\Phi(\{a_p\})$. 
For each positive even integer $k\ge 4$, let 
$$E^1_{k}(\tau)=\ds\sum_{(c,d)\in \Z^2\setminus\{(0,0)\}\atop (c,d)=1}(c\tau+d)^{-k},
$$ 
which is the Eisenstein series of 
weight $k$ with respect to $SL_2(\Z)$. 

\begin{defin}\label{family} For a sufficiently large $k_0$, a compatible family of Eisenstein series is a family of elliptic modular forms, for even 
integer $k'\geq k_0$
$$g_{k'}(\tau)=b_{k'}(0)+\sum_{N\in\Q_{>0}} N^{\frac{k'-1}{2}} b_{k'}(N) q^N,\ q=\bold{e}(\tau), 
$$
satisfying the following three conditions:
\begin{enumerate}
\item $g_{k'}\in\mathcal V(E^1_{k'})$ for all $k'\geq k_0$
\item for each $N\in\Q_+^\times$, there exists $\Phi_N\in\mathcal R$ such that $b_{k'}(N)=\Phi_N(\{p^{\frac{k'-1}{2}}\}_p)$.
\item there exists a congruence subgroup $\Gamma\subset SL_2(\Z)$ such that $g_{k'}\in M_{k'}(\Gamma)$ for all $k'\geq k_0$. 
Here $ M_{k'}(\Gamma)$ stands for the space of elliptic modular forms of weight $k$ with respect to $\Gamma$.
\end{enumerate}
\end{defin}

Then by Lemma 10.2 of \cite{Ik2}, we have
\begin{lemma}\label{ikeda-lem} Let $f(\tau)=\ds\sum_{n=1}^\infty c(n)q^n$ be a Hecke eigenform of weight $k$ with respect to $SL_2(\Z)$ with 
$c(p)=p^{\frac{k-1}{2}}(\alpha_p+\alpha^{-1}_p)$. 
Assume that there is a finite dimensional representation $(u,\C^d)$ of $SL_2(\Z)$ and 
$$\vec{\Phi}_N:={}^t(\Phi_{1,N},\ldots,\Phi_{d,N})\in\mathcal R^d,\ N\in\Q_{>0}
$$ 
satisfying the following two conditions: 
\begin{enumerate}
\item there exists a vector valued modular form $\vec{g}_{k'}={}^t(g_{1,k'},\ldots,g_{d,k'})$ which has  
$$\vec{g}_{k'}(\tau)=\vec{b}_{k'}(0)+\sum_{N\in\Q_{>0}} N^{\frac{k'-1}{2}} \vec{b}_{k'}(N) q^n,\  (\vec{b}_{k'}(N)=
{}^t(b_{1,k'}(N),\ldots,b_{d,k'}(N)),\ N\in \Q_{\ge 0})$$ of weight $k'$ with type $u$ 
for each sufficiently large even integers $k'$, hence this means that 
$$\vec{g}_{k'}(\tau)|_{k'}[\gamma]:={}^t(g_{1,k'}|_{k'}[\gamma],\ldots,g_{d,k'}|_{k'}[\gamma])=
u(\gamma)\vec{g}_{k}(\tau)\ {\rm for\ any}\ \gamma\in SL_2(\Z),
$$ 
\item each component $g_{i,k'},\ (1\le i\le d)$ of $\vec{g}_{k'}(\tau)$ is a compatible family of Eisenstein series such that $$b_{i,k'}(N)=\Phi_{i,N}(\{p^{\frac{k'-1}{2}}\}_p).
$$ 
\end{enumerate}
Then $\vec{h}(\tau):=\ds\sum_{N\in\Q_{>0}} N^{\frac{k-1}{2}} \vec{\Phi}_{N}(\{\alpha_p\}_p) q^N$ is 
a vector valued modular form of weight $k$ with type $u$, hence it satisfies 
$$\vec{h}(\tau)|_{k}[\gamma]=u(\gamma)\vec{h}\ {\rm for\ any}\ \gamma\in SL_2(\Z).
$$
\end{lemma}

\section{Construction of cusp forms on the exceptional domain}\label{construction}
In this section we shall prove our main theorem. The strategy is the same as in \cite{Ik1}, \cite{Ik2}, and \cite{Yamana}. 

For any positive integer $k\ge 10$, let $f(\tau)=\ds\sum_{n=1}^\infty c(n)q^n$ be a Hecke eigenform of 
weight $2k-8$ with respect to $SL_2(\Z)$ with 
$c(p)=p^{\frac{2k-9}{2}}(\alpha_p+\alpha^{-1}_p)$. 
Let us formally define a function on $\frak T$ by  
$$F(Z)=\sum_{T\in \frak J(\Z)_+}A_F(T)\bold{e}((T,Z)),\quad 
A_F(T)=\det(T)^{\frac{2k-9}{2}} \prod_p \widetilde{f}_T^p(\alpha_p) ,\ Z\in \frak T.
$$ 
Then we prove the following:
\begin{theorem}\label{main-thm} $F(Z)$ is a non-zero cusp form of weight $2k$ with respect to $\Gamma$.
\end{theorem}

\begin{remark} \label{integer}
If $f$ has integer Fourier coefficients, then $F$ also has integer Fourier coefficients. 
Just observe from Corollary \ref{fourier2} that $\widetilde{f}_T^p(X)=X^d+X^{-d}+a_1(X^{d-2}+X^{-(d-2)})+\cdots+a_{\frac {d-1}2} (X+X^{-1})$ if $d$ is odd, and $\widetilde{f}_T^p(X)=X^d+X^{-d}+a_1(X^{d-2}+X^{-(d-2)})+\cdots+a_{\frac d2}$ if $d$ is even, where $d=\text{ord}_p(det(T))$ and $a_i$'s are integers. 
\end{remark} 

First of all we shall prove the convergence of $F(Z)$:
\begin{lemma}\label{conv} The series $F(Z)$ is absolutely and uniformly convergent on any compact domain of $\frak T$. 
\end{lemma}
\begin{proof}
It is well-known that $|\alpha_p|=1$. By definition, $\widetilde{f_T^p}(X)=X^d f_T^p(X^{-2})$, and $f_T^p$ is a monic polynomial with integer coefficients of degree $d={\rm ord}_p(\det(T))$, i.e.,
$f_T^p(X)=X^d+a_1 X^{d-1}+\cdots+a_{d-1}X+a_d$. Let $M=\max\{|a_1|,...,|a_d|\}$. 
We use the identity from \cite{Ka}, page 187,
$$(1-p^{-s})^{-1} S_p(T)=\sum_{m=0}^\infty \alpha_m(T)p^{-ms},\quad \alpha_m(T)=\sum_X \omega_m^{(T,X)},
$$
where $X\in \Lambda(3)_p/p^m\Lambda(3)_p$ and $\tau_i(X)\equiv 0$ (mod $p^{m(i-1)}$) for $2\leq i\leq 3$, and $2m\leq d$.
Hence $|\alpha_m(T)|\leq p^{27m}$. We also have (\cite{Ka}, page 197)
$$S_p(T)=(1-p^{-s})(1-p^{4-s})(1-p^{8-s})f_T^p(p^{9-s}).
$$
Hence 
$$f_T^p(X)=(1-p^{-5}X)^{-1}(1-p^{-1}X)^{-1}\sum_{m=0}^\infty \alpha_m(T)p^{-9m} X^m.
$$
So $M\leq (d+1)^2 p^{18m}$. By the trivial estimate, $d+1\leq p^d$, hence we have $M\leq p^{2d}p^{9d}=p^{11d}.$
Therefore, $|\widetilde{f_T^p}(\alpha_p)|\leq (d+1)M\leq p^{12d}.$
Hence 
$$|A_F(T)|\leq \det(T)^{k+12-\frac 12}.
$$

Now we use the fact from \cite{B}, page 538, for $l>8$, 
$$\int_{R_3^+(\Bbb R)} \det(X)^{l-9} e^{2\pi(X,Y)}\, dX=\pi^{12} (2\pi i)^{-3l} \prod_{n=0}^2 \Gamma(l-4n) \det(Y)^{-l}.
$$
where $dX$ is the ordinary Euclidean measure.
Hence 
$$\left|\sum_{T\in \frak J(\Bbb Z)_+} A_F(T) e^{2\pi i (T,Z)}\right|\leq \sum_{T\in \frak J(\Bbb Z)_+} \det(T)^{k+12-\frac 12} e^{2\pi (T,Y)}\leq 
\int_{R_3^+(\Bbb R)} \det(X)^{k+12-\frac 12} e^{2\pi (X,Y)}\, dX,
$$
converges.
\end{proof}
Clearly $F(Z+N)=F(Z)$ for $N\in \bold N(\Z)$. Also $F(\gamma Z)=F(Z)$ for $\gamma\in\bold M'(\Z)$.
Thanks to Lemma \ref{generator}, it is enough to prove 
\begin{equation}\label{inv}
F(\iota_{e_3} Z)=j(\iota_{e_3},Z)^{2k} F(Z).
\end{equation}

We prove (\ref{inv}) by using results of previous sections: Fix $S\in \frak J_2(\Z)_{+}$. 
Since $F(Z)$ is invariant under $\G\cap \bold P(\Q)$ as mentioned above and is holomorphic by Lemma \ref{conv}, then by 
Remark \ref{important-rmk}, one has the Fourier-Jacobi expansion: 
\begin{equation}\label{FJ1}
F\left(\begin{array}{cc}
W & u \\
{}^t\bar{u} & \tau 
\end{array}\right)=\sum_{S\in  \frak J_2(\Z)_+}F_S(\tau,u)\bold{e}((S,W)), 
\end{equation} 
\begin{equation}\label{FJ2}
F_S(\tau,u)=\sum_{\xi\in \Xi(S)}F_{S,\xi}(\tau) \theta^S_{\varphi_\xi}(\tau,u), 
\end{equation}
and 
\begin{equation}\label{FJ3} 
\begin{array}{ll}
F_{S,\xi}(\tau)&=\ds\sum_{N\in \Z \atop N-\sigma_S(\xi,\xi)\ge 0}
A_F(S_{\xi,N})\bold{e}((N-\sigma_S(\xi,\xi))\tau),\ S_{\xi,N}:=\left(\begin{array}{cc}
S & S\xi \\
{}^t\bar{\xi}S & N 
\end{array}\right) \\
&=\ds\sum_{N\in \Z \atop N-\sigma_S(\xi,\xi)\ge 0}\det(S_{\xi,N})^{\frac{2k-9}{2}}
\prod_p \widetilde{f}_{S_{\xi,N}}^p(\alpha_p)\bold{e}((N-\sigma_S(\xi,\xi))\tau)\\
&=\det(S)^{\frac{2k-9}{2}}\ds\sum_{N\in \Z \atop N-\sigma_S(\xi,\xi)\ge 0}(N-\sigma_S(\xi,\xi))^{\frac{2k-9}{2}}
\prod_p \widetilde{f}_{S_{\xi,N}}^p(\alpha_p)\bold{e}((N-\sigma_S(\xi,\xi))\tau).
\end{array}
\end{equation}
For the last equality above, we used the formula $\det(S_{\xi,N})=\det(S)(N-\sigma_S(\xi,\xi))$ by using (\ref{formula}). 
The condition (\ref{inv}) is equivalent to claiming that $F_S(\tau,u)\in J_{k,S}(\Gamma_J)$ for any $S\in \frak J_2(\Z)_+$. 
Therefore for each fixed $S\in \frak J_2(\Z)_+$, we have only to check the condition 
\begin{equation}\label{inv1}
F_S|_{k,S}[w_1](\tau,u)=F_S(\tau,u)\ {\rm for}\  
w_1=\left(\begin{array}{cc}
0 & 1 \\
-1 & 0   
\end{array}\right).
\end{equation}
By (2,1), p.124 of \cite{takase}, for each $\gamma\in SL_2(\Z)$, there exists a unitary matrix $u_S(\gamma)=
(u_S(\gamma)_{\xi \eta})_{\xi,\eta\in \Xi(S)}$ such that 
\begin{equation}\label{inv2}
\theta^S_{\varphi_\xi}|_{k,S}[\gamma](\tau,u)=\sum_{\eta\in \Xi(S)}u_S(\gamma)_{\xi\eta}\theta^S_{\varphi_\eta}(\tau,u).
\end{equation}
Further there exists a positive integer $\Delta_S$ depending on $S$ such that $u_S$ is trivial on $\Gamma(\Delta_S)\subset SL_2(\Z)$.
Since $\{ \theta^S_{\varphi_\xi}\ |\  \xi\in \Xi(S) \}$ are linearly independent over $\C$, 
to check (\ref{inv1}), with (\ref{inv2}) it suffices to prove that $\{F_{S,\xi}\}_{\xi\in \Xi(S)}$ is a vector 
valued modular form of weight $2k$ with type $u_S$. 

For a sufficiently large positive integer $k'$, 
we now turn to consider $(S,\xi)$-component $(\mathcal{E}_{2k',0})_{S,\xi}$ of the classical 
Eisenstein series  
\begin{eqnarray*} \mathcal{E}_{2k',0}(Z):&=&\frac{1}{C_{2k'}}E_{2k',0}(Z)=\sum_{T\in \frak J(\Z)_+}\widetilde{a}_{2k'}(T)\bold{e}((T,Z)),\\
\widetilde{a}_{2k'}(T)&=&\det(T)^{\frac{2k'-9}{2}}\prod_{p|\det(T)}\widetilde{f}^p_T(p^{\frac{2k'-9}{2}}),
\end{eqnarray*}
on $\frak T$, where $C_{2k'}$ is the constant in Theorem \ref{fourier1}. Then one has 
$$
\begin{array}{l}
\det(S)^{-\frac{2k'-9}{2}}(\mathcal{E}_{2k',0})_{S,\xi}(\tau)
=\det(S)^{-\frac{2k'-9}{2}}
\ds\sum_{N\in \Z \atop N-\sigma_S(\xi,\xi)\ge 0}
\widetilde{a}_{2k'}(\left(\begin{array}{cc}
S & S\xi \\
{}^t\bar{\xi}S & N 
\end{array}\right))\bold{e}((N-\sigma_S(\xi,\xi))\tau)\\
=
\ds\sum_{N\in \Z \atop N-\sigma_S(\xi,\xi)\ge 0}(N-\sigma_S(\xi,\xi))^{\frac{2k'-9}{2}}
\prod_{p|\det(S_{\xi,N})}\widetilde{f}^p_{S_{\xi,N}}(p^{\frac{2k'-9}{2}})\bold{e}((N-\sigma_S(\xi,\xi))\tau)\\
\end{array}
$$
Then by Lemma \ref{relation}, Lemma \ref{fourier2}, and Theorem \ref{key-thm}, 
$\{\det(S)^{-\frac{2k'-9}{2}}(\mathcal{E}_{2k',0})_{S,\xi}\}_{k'\gg0}$ makes up a family of Eisenstein series. 
Here we use $u_S|_{\Gamma(\Delta_S)}=1$ to check the third condition of Definition \ref{family}.  
Applying Lemma \ref{ikeda-lem} with (\ref{FJ3}), one can conclude that
$$F_{S,\xi}=\det(S)^{\frac{2k-9}{2}}\ds\sum_{n\in\Z_{>0}\atop 
n=N-\sigma_S(\xi,\xi),\ N\in \Z} n^{\frac{2k-9}{2}} \prod_{p|\det(S_{\xi,N})}\widetilde{f}^p_{S_{\xi,N}}(\alpha_p) q^n,
$$ 
is 
a vector valued modular form of weight $2k$ with type $u_S$. 
Since $A_F(1)=1$, $F(Z)$ is not identically zero. 
This completes the proof. 

\section{Hecke operators}
Karel \cite{Ka1} defined Hecke operators and showed that the Eisenstein series are eigenfunctions. We review his results, and show that the cusp form on $\frak T$ constructed in the previous section is a Hecke eigenform.

Let $I_\bold W$ be the identity operator on $\bold W$. Let $\bold Z$ be the central torus of $GL(\bold W)$, i.e., for any field $K$,
$$\bold Z_K=\{ \lambda I_{\bold W} |\, \text{$\lambda\in K$, $\lambda\ne 0$}\}.
$$
Let $\widetilde{\bold G}=Z\cdot \bold G$. Then $\widetilde{\bold G}$ is a $\Q$-group. Define a rational character $\mu$ on $\widetilde{\bold G}$ by 
$$\{gw_1,gw_2\}=\mu(g)\{w_1,w_w\},\quad \text{for all $w_1, w_2\in \bold W$}.
$$
Then $\mu$ is defined over $\Q$, and $Q(gw)=\mu(g)^2 Q(w)$. 
Let
$S$ be the connected component of the Lie group $\widetilde{\bold G}(\R)$ containing the identity element of $\bold G(\R)$. 
Define 
$$\Psi=\{g\in S |\, g \bold W_{\frak o}\subset \bold W_{\frak o}\}.
$$ 
Since $S$ is a connected component containing the identity element, $\mu(g)>0$ for all $g\in \Psi$. 
Recall that $e=(0,1,0,0)$ and $e'=(0,0,0,1)$ are elements of $\bold W_{\frak o}$, and $\{w_1,w_2\}\in\Bbb Z$ for all $w_1, w_2\in 
\bold W_{\frak o}$.
Hence $\mu(g)=\{ge,ge'\}\in\Bbb Z$. 
Hence we can define,
for each $m\in\Bbb Z$, $m>0$, 
$$\Psi_m=\{g\in\Psi |\, \mu(g)=m\},
$$
and $\Psi=\bigcup_{m=1}^\infty \Psi_m$. 

Fix $k$. If $\rho=z\rho'\in\widetilde{\bold G}(\R)=Z(\R)_+\cdot \bold G(\R)$ and $F$ is a function on $\frak T$, 
let $F(Z)|[\rho]_k=F(\rho' Z) j(\rho',Z)^{-k}$.
If $F$ is holomorphic, then $F|[\rho]_k=F$ for all $\rho\in \Gamma$ precisely when $F$ is a modular form of weight $k$.
Let $F$ be a modular form on $\frak T$ of weight $k$, and define
$$T(m)\cdot F=\sum_{\rho\in \Gamma\backslash\Gamma \Psi_m \Gamma} \rho\cdot F.
$$
Actually, in \cite{Ka1} Karel used $J(g,Z)=j(g,Z)^{-18}$ as an automorphy factor. However,
his result works for $j(g,Z)$ in the exactly same way. 

For later purpose in connection with representation theory, we shall modify Hecke operators for $\bold G(\Q)$.  
For any element $H\in \bold G(\Q)$, we define a modified action of $H$ on $F$ by 
$$H\star F =v_H(\Gamma)^{-\frac{k}{36}}\sum_{\rho\in \Gamma\backslash\Gamma H \Gamma} \rho\cdot F,\ 
v_H(\Gamma):=[H\G{H}^{-1}:\Gamma].$$
    
Then we have
\begin{prop}\label{eisen-hecke} $E_{l,0}(Z)$ is a Hecke eigenform for each Hecke operator $T(m)$. 
In particular it is also an eigenform for any $H\in \bold G(\Q)$ with respect to the modified action $\star$. 
\end{prop}

For any positive integer $k\ge 10$, let $f(\tau)=\ds\sum_{n=1}^\infty c(n)q^n$ be a Hecke eigenform of 
weight $2k-8$ with respect to $SL_2(\Z)$ with 
$c(p)=p^{\frac{2k-9}{2}}(\alpha_p+\alpha^{-1}_p)$. Let $F(Z)=\ds\sum_{T\in \frak J(\Z)_+} A_F(T)\bold e((T,Z)),\ Z\in \frak T$ be the modular form on $\frak T$ which is constructed from $f$ in 
previous section. Then by imitating Ikeda's idea in Section 11 of \cite{Ik3} we will prove that 
$F(Z)$ is 
a Hecke eigenform for any $H\in \bold G(\Q)$. 
Recall the normalized Eisenstein series $$\mathcal{E}_{2k',0}(Z)=
\sum_{T\in \frak J(\Z)_+}\widetilde{a}_{2k'}(T)\bold{e}((T,Z)),\quad 
\widetilde{a}_{2k'}(T)=\det(T)^{\frac{2k'-9}{2}}\prod_{p|\det(T)}\widetilde{f}^p_T(p^{\frac{2k'-9}{2}})$$ of weight $2k'$ which is also 
an eigenform for any $H\in \bold G(\Q)$ by Proposition \ref{eisen-hecke}. 
For each $H\in \bold G(\Q)$, by using $\bold G(\Q)=P(\Q)\bold G(\Z)$ (see page 532, line -4 of \cite{B}), 
one can choose $\{p_{n_i}\cdot m_i\}_{i=1}^r,\ n_i\in \frak J(\Q),\ m_i\in M(\Q)_+$ as the complete representatives of 
$\Gamma\backslash\Gamma H\Gamma$. Here $M(\Q)_+$ is the subset of $M(\Q)$ consisting of $g$ with $\nu(g)>0$. 
Then it is easy to see that $v_H(\Gamma)^{-\frac{1}{36}}=\nu(m_i)^{\frac{1}{2}}$ for each $i$.
Henceforth we settle the convention that 
for each $T\in \frak J(\Z)_+$ and each $m\in M(\Q)$,  put 
$\widetilde{f}_{mT}(\{X_p\}_p)=0$ if $mT\not\in \frak J(\Z)_+$.  
Then one has 
$$\begin{array}{ll}
H\star \mathcal{E}_{2k',0}(Z)&=v_H(\Gamma)^{-\frac{2k'}{36}}\ds\sum_{i=1}^r (p_{n_i}\cdot m_i)\cdot \mathcal{E}_{2k',0}(Z)\\
&=
\ds\sum_{T\in \frak J(\Z)_+} \sum_{i=1}^r\nu(m_i)^{k'} \det(T)^{\frac{2k'-9}{2}}
\prod_{p|\det(T)}\widetilde{f}^p_T(p^{\frac{2k'-9}{2}}) \bold e((T,m_i Z+n_i))\\
&({\rm use}\ (T,m_iZ)=((m^\ast_i)^{-1}T,Z)\ {\rm by}\ (\ref{identity})\ {\rm and}\ \det(m^\ast_iT)=\nu(m_i)^{-1}\det(T).)  \\
&=
\ds\sum_{T\in \frak J(\Z)_+} \sum_{i=1}^r\nu(m_i)^{\frac{9}{2}} \bold e((m^\ast_i T,n_i))
\det(T)^{\frac{2k'-9}{2}}\prod_{p|\det(m^\ast_i T)}\widetilde{f}^p_{m^\ast_i T}(p^{\frac{2k'-9}{2}}) \bold e((T,Z)). 
\end{array}$$
From this, the $T$-th Fourier coefficient of $H\star \mathcal{E}_{2k',0}$ is 
$$\sum_{i=1}^r\nu(m_i)^{\frac{9}{2}} \bold e((m^\ast_i T,n_i))
\det(T)^{\frac{2k'-9}{2}}\prod_{p|\det(m^\ast_i T)}\widetilde{f}^p_{m^\ast_i T}(p^{\frac{2k'-9}{2}}).$$
Put $$\alpha_H(X):=\sum_{i=1}^r\nu(m_i)^{\frac{9}{2}} \bold e((m^\ast_i E,n_i))
\prod_{p|\det(m^\ast_i  E)}\widetilde{f}^p_{m^\ast_i  E}(X_p),\ X=\{X_p\}_p$$
which defines an element of $\otimes'_p\C[X_p,X^{-1}_p]$. Here $E={\rm diag}(1,1,1)\in \frak J(\Q)$.  
Noting $\widetilde{a}_{2k'}(E)=1\not=0$, by Proposition \ref{eisen-hecke} one has 
$$\alpha_H(\{p^{\frac{2k'-9}{2}}\}_p)\prod_{p|\det(T)}\widetilde{f}^p_T(p^{\frac{2k'-9}{2}})
=\sum_{i=1}^r\nu(m_i)^{\frac{9}{2}} \bold e((m^\ast_i T,n_i))
\prod_{p|\det(m^\ast_i T)}\widetilde{f}^p_{m^\ast_i T}(p^{\frac{2k'-9}{2}}).$$
By Lemma 10.1 of \cite{Ik1}, one has the equality in $\otimes'_p\C[X_p,X^{-1}_p]$ 
$$\alpha_H(X)\prod_{p|\det(T)}\widetilde{f}^p_T(X_p)
=\sum_{i=1}^r\nu(m_i)^{\frac{9}{2}} \bold e((m^\ast_i T,n_i))
\prod_{p|\det(m^\ast_i T)}\widetilde{f}^p_{m^\ast_i T}(X_p),\ X=\{X_p\}_p.$$   
Then one has 
$$
\begin{array}{rl}
H\star F(Z)=&
\ds\sum_{T\in \frak J(\Z)_+} \sum_{i=1}^r\nu(m_i)^{\frac{9}{2}} \bold e((m^\ast_i T,n_i))
\det(T)^{\frac{2k'-9}{2}}\prod_{p|\det(m^\ast_i T)}\widetilde{f}^p_{m^\ast_i T}(\alpha_p) \bold e((T,Z))\\
=&
\ds\sum_{T\in \frak J(\Z)_+ }\alpha_H(\{\alpha_p\}_p)
\det(T)^{\frac{2k'-9}{2}}\prod_{p|\det(T)}\widetilde{f}^p_{T}(\alpha_p) \bold e((T,Z))
=\alpha_H(\{\alpha_p\}_p)F(Z). 
\end{array}
$$
Hence we have proved the following:
\begin{theorem}\label{hecke-eigen} $F(Z)$ is a Hecke eigenform for $\bold{G}(\Q)$ with respect to the action $\star$.
\end{theorem}

\section{The degree $56$ standard L-function}
In this section we will compute the standard L-function of Hecke eigenforms constructed in the previous section 
and the Eisenstein series respectively. 
Let $F=F(Z)$ be the cusp form in Theorem \ref{hecke-eigen} and $\widetilde{F}$ be the automorphic form on 
$\bold G(\Bbb A)$ attached to $F$ (see (\ref{auto-form})). 
Let $\pi_F$ be the cuspidal representation of $\bold G(\Bbb A)$ attached to $\widetilde{F}$. 
Since $F$ is a Hecke eigenform, one has the decomposition $\pi_F=\pi_\infty\otimes \otimes_p' \pi_p$. 
Then $\pi_\infty$ is a holomorphic discrete series of the lowest weight $2k$ associated to $-2k\varpi_7$ in the notation of \cite{Bour}.
We note that $-2k\varpi_7$ parametrizes a holomorphic discrete series when $2k>17$ (cf. \cite{knapp}, page 158).
Since $\pi_p$ is unramified for each prime $p$, it has a spherical vector whose Hecke eigenvalue for 
each element of $\bold{G}(\Q_p)$ coincides with 
that of a spherical vector in ${\rm Ind}_{\bold P(\Q_p)}^{\bold{G}(\Q_p)}\: |\nu(g)|^{2s_p}$ where $p^{s_p}=\alpha_p$. (This is clear from 
the proof of Theorem \ref{hecke-eigen}. Notice $2s_p$, not $s_p$. We can see it from Corollary 6.2 and Proposition 6.4. We are replacing $\frac {2k-9}2$ by $s_p$ in Corollary 6.2.)    
Then by Proposition 2.2.2 of \cite{casselman} and Proposition \ref{W}, 
$$\pi_p\simeq {\rm Ind}_{\bold P(\Q_p)}^{\bold{G}(\Q_p)}\: |\nu(g)|^{2s_p},
$$ 
for any finite place $p$. 

In order to compute the standard $L$-function of $\pi_F$, we use Langlands-Shahidi method. 
Since $\bold G(\Bbb Q_p)$ is the split group of type $E_7$, we can compute its local $L$-factor. 
In the notation of \cite{kim1}, it is in section 2.7.8. We consider the split exceptional group of type $E_8$, and its parabolic subgroup $R$ whose Levi subgroup is $GE_7$, and its Borel subgroup $B$.
Since 
$${\rm Ind}_{R(\Q_p)}^{E_8(\Q_p)}\: \pi_p\otimes \exp(s\tilde\alpha, H_R(\,))={\rm Ind}_{B(\Q_p)}^{E_8(\Q_p)}\: \exp(\chi, H_B(\, )),
$$
where $\tilde\alpha=e_1-e_9$, and 
$\chi=s(e_1-e_9)+s_p(-e_1+2e_2-e_9)+(8 e_3+7 e_4+ 6 e_5+ 5 e_6 +4 e_7+3 e_8)$. Here $\rho_{E_6}=8 e_3+7 e_4+ 6 e_5+ 5 e_6 +4 e_7+3 e_8$ is the half-sum of positive roots of $E_6$. Then one can see that the unipotent radical of $R$ is generated by 57 roots
\begin{eqnarray*}
&& \text{$e_i-e_9$ for $i=1,...,8$, and $e_1-e_j$ for $j=2,...,8$}\\ 
&&\text{$e_1+e_j+e_k$ for $2\leq j<k\leq 8$, and $-(e_i+e_j+e_9)$ for $2\leq i<j\leq 8$}.
\end{eqnarray*}
Then $e_1-e_9$ gives rise to $1-p^{-2s}$, and the remaining 56 roots give rise to the following local factors:
\begin{eqnarray*}
&& e_1-e_3, 1-\alpha_p p^{8-s}; \ e_1+e_2+e_8, 1-\alpha_p^{-1} p^{8-s}; \ e_3-e_9, 1-\alpha_p^{-1} p^{-8-s}; \ -(e_2+e_8+e_9), 1-\alpha_p p^{-8-s}\\
&& e_1-e_4, 1-\alpha_p p^{7-s}; \ e_1+e_2+e_7, 1-\alpha_p^{-1} p^{7-s}; \ e_4-e_9, 1-\alpha_p^{-1} p^{-7-s};\ -(e_2+e_7+e_9), 1-\alpha_p p^{-7-s}\\
&& e_1-e_5, 1-\alpha_p p^{6-s}; \ e_1+e_2+e_6: 1-\alpha_p^{-1} p^{6-s}; \ e_5-e_9: \ 1-\alpha_p^{-1} p^{-6-s};\ -(e_2+e_6+e_9), 1-\alpha_p p^{-6-s}\\
&& e_1-e_6: \ 1-\alpha_p p^{5-s}; \ e_1+e_2+e_5, 1-\alpha_p^{-1} p^{5-s}; \ e_6-e_9, 1-\alpha_p^{-1} p^{-5-s};\ -(e_2+e_5+e_9), 1-\alpha_p p^{-5-s}\\
&& e_1-e_7, e_1+e_7+e_8, 1-\alpha_p p^{4-s}; \ e_1+e_2+e_4, -(e_3+e_4+e_9), 1-\alpha_p^{-1} p^{4-s} \\
&& e_1-e_8, e_1+e_6+e_8, 1-\alpha_p p^{3-s}; \ e_1+e_2+e_3,  -(e_3+e_5+e_9), 1-\alpha_p^{-1} p^{3-s} \\
&& e_1+e_6+e_7,  e_1+e_5+e_8, 1-\alpha_p p^{2-s}; \ -(e_3+e_6+e_9), -(e_4+e_5+e_9), 1-\alpha_p^{-1} p^{2-s} \\
&& e_1+e_5+e_7,  e_1+e_4+e_8, 1-\alpha_p p^{1-s}; \ -(e_3+e_7+e_9), -(e_4+e_6+e_9), 1-\alpha_p^{-1} p^{1-s} \\
&& e_1+e_3+e_8,  e_1+e_4+e_7, 1-\alpha_p p^{-s}; \ -(e_3+e_8+e_9), -(e_4+e_7+e_9), 1-\alpha_p^{-1} p^{-s} \\
&& e_1+e_3+e_7,  e_1+e_4+e_6, 1-\alpha_p p^{-1-s}; \ -(e_4+e_8+e_9),\ -(e_5+e_7+e_9), 1-\alpha_p^{-1} p^{-1-s} \\
&& e_1+e_3+e_6,  e_1+e_4+e_5, 1-\alpha_p p^{-2-s}; \ -(e_5+e_8+e_9),\ -(e_6+e_7+e_9), 1-\alpha_p^{-1} p^{-2-s} \\
&& e_1+e_3+e_5,  -(e_2+e_3+e_9),  1-\alpha_p p^{-3-s}; \ e_8-e_9, -(e_6+e_8+e_9), 1-\alpha_p^{-1} p^{-3-s} \\
&& e_1+e_3+e_4,  -(e_2+e_4+e_9), 1-\alpha_p p^{-4-s}; \ e_7-e_9, -(e_7+e_8+e_9), 1-\alpha_p^{-1} p^{-4-s} \\
&& e_1-e_2, 1-\alpha_p^3 p^{-s}; \ e_2-e_9, 1-\alpha_p^{-3} p^{-s}; \ e_1+e_5+e_6, 1-\alpha_p p^{-s}; \ -(e_5+e_6+e_9), 1-\alpha_p^{-1} p^{-s}
\end{eqnarray*}
Hence we have the degree 56 local $L$-function:
\begin{eqnarray*}
&& (1-\alpha_p p^{-s})^2 (1-\alpha_p^{-1} p^{-s})^2
\prod_{i=0}^3 (1-\alpha_p^{3-2i} p^{-s}) \\
&&\cdot \prod_{i=5}^8 (1-\alpha_p p^{\pm i-s})(1-\alpha_p^{-1} p^{\pm i-s})
\prod_{i=1}^4 (1-\alpha_p p^{\pm i-s})^2(1-\alpha_p^{-1} p^{\pm i-s})^2.
\end{eqnarray*}
Therefore, we have proved

\begin{theorem} The degree $56$ standard L-function $L(s,\pi_F,St)$ of $\pi_F$ is given by
$$L(s,\pi_F,St)=L(s,{\rm Sym}^3 \pi_f)L(s,\pi_f)^2 \prod_{i=1}^4 L(s\pm i,\pi_f)^2 \prod_{i=5}^8 L(s\pm i,\pi_f),
$$
where $L(s,{\rm Sym}^3 \pi_f)$ is the third symmetric power $L$-function.
\end{theorem}
Let $\Gamma_\C(s)=2(2\pi)^{-s}\Gamma(s)$.
Then the local $L$-factor at $\infty$ is given by
$$L(s,\pi_\infty,St)=\Gamma_\C(s+\tfrac {3(2k-9)}2)\Gamma_\C(s+\tfrac {2k-9}2) \Gamma_\C(s+\tfrac {2k-9}2)^2 
\prod_{i=1}^4 \Gamma_\C(s+\tfrac {2k-9}2\pm i)^2 \prod_{i=5}^8 \Gamma_\C(s+\tfrac {2k-9}2\pm i),
$$
and the completed $L$-function satisfies the functional equation
$$\Lambda(s,\pi_F,St)=L(s,\pi_\infty,St)L(s,\pi_F,St)=- \Lambda(1-s,\pi_F,St),
$$
Note that the root number is $-1$, since the root number of $L(s, {\rm Sym}^3 \pi_f)$ is $-1$ \cite{CM}.

We have also proved that the standard $L$-function $L(s,E_{2l,0}(Z),St)$ of $E_{2l,0}(Z)$ is 
\begin{theorem}
\begin{eqnarray*} 
&& L(s,E_{2l,0}(Z),St)=\zeta(s+l-\tfrac {9}2)^2 \zeta(s-l+\tfrac {9}2)^2 \zeta(s-3l+\tfrac {27}2) \zeta(s-l+\tfrac {9}2) \zeta(s+l-\tfrac {9}2)
\zeta(s+3l-\tfrac {27}2) \\
&&\cdot \prod_{i=5}^8 \zeta(s\pm i-l+\tfrac {9}2) \zeta(s\pm i +l-\tfrac {9}2) \prod_{i=1}^4 \zeta(s\pm i-l+\tfrac {9}2)^2 \zeta(s\pm i+l-\tfrac {l}2)^2.
\end{eqnarray*}
\end{theorem}

\begin{remark} 
We write the degree 56 standard $L$-function of $\pi_F$ as
$$L(s,\pi_F,St)=L(s,Sym^3 \pi_f)\prod_{i=-4}^4 L(s+i,\pi_f) \prod_{i=-8}^8 L(s+i,\pi_f).
$$
This suggests the following parametrization of $\pi_F$: 
Let $\mathcal L$ be the (hypothetical) Langlands group over $\Q$, and let $\rho_f : \mathcal L\lra SL_2(\C)$ be the 2-dimensional irreducible representation of $\mathcal L$ corresponding to $\pi_f$. 
Let $Sym^n$ be the irreducible $(n+1)$-dimensional representation of $SL_2(\C)$. Note that $Im(Sym^3)\subset Sp_4(\C)$. Then 

$\rho_f\boxtimes Sym^{16}: \mathcal L\times SL_2(\Bbb C)\longrightarrow Sp_{34}(\Bbb C)$, and
$\rho_f\boxtimes Sym^{8}: \mathcal L\times SL_2(\Bbb C)\longrightarrow Sp_{18}(\Bbb C)$.

\noindent Let $Sym^3\rho_f: \mathcal L\times SL_2(\Bbb C)\longrightarrow Sp_4(\Bbb C)$ be the parameter of $Sym^3\pi_f$, where it is trivial on $SL_2(\Bbb C)$. Consider the parameter
$$
\rho=Sym^3\rho_f\oplus \rho_f\boxtimes Sym^{16}\oplus \rho_f\boxtimes Sym^{8}: \mathcal L\times SL_2(\Bbb C)\longrightarrow 
Sp_4(\Bbb C)\times Sp_{34}(\Bbb C)\times Sp_{18}(\Bbb C)\subset Sp_{56}(\Bbb C).
$$

Note that $E_7(\Bbb C)\subset Sp_{56}(\Bbb C)$. We expect that $\rho$ will factor through $E_7(\Bbb C)$, and give rise to a parameter
$\rho: \mathcal L\times SL_2(\Bbb C)\longrightarrow 
E_7(\Bbb C)$, which parametrizes $\pi_F$.
\end{remark}

\section{Appendix}
In this appendix we will compute the discriminant of some quadratic forms from Section 4.2 and prove the orthogonal relation of theta functions in the proof of Lemma 5.9.
 
Let $S=\begin{pmatrix} a&u\\ \bar u&b\end{pmatrix}\in \frak J_2(K)$ where $K$ is a field whose characteristic is different from 2,3. 
Recall $\det(S)=ab-N(u)$ and the quadratic form $\lambda_S(x,y)=\frac{1}{2}(S,x{}^t\bar{y}+y{}^t\bar{x})$ on $X(K)$ (see (\ref{quadratic})). 
We denote by $\disc(\lambda_S)$ the discriminant of the quadratic form $\lambda_S$, i.e.,
the determinant of the representation matrix of $\lambda_S$. 
Then we have 
\begin{lemma}\label{det} ${\rm disc}(\lambda_S)=\det(S)^8$.
\end{lemma}
\begin{proof}
Let $S=\begin{pmatrix} a&u\\ \bar u&b\end{pmatrix}$, where $a,b\in K$ and $u\in \frak C_K$.
Let $x=\begin{pmatrix} x_1\\ x_2\end{pmatrix}$, $y=\begin{pmatrix} y_1\\ y_2\end{pmatrix}$where $x_1,x_2, y_1,y_2\in 
\frak C_K$.
Let 
$$\lambda_S(x,y)=\frac{1}{2}(S, x {}^t \bar y+y {}^t \bar x)=\frac 12( a(x_1 \bar{y_1}+y_1\bar{x_1})+b(x_2\bar{y_2}+y_2\bar{x_2})
+u( x_2\bar{y_1}+y_2\bar{x_1})+(x_1\bar{y_2}+y_1\bar{x_2})\bar u),
$$
be the bilinear form given by $S$.

For $x\in \frak C_K$, let $x=x_0e_0+\cdots+x_7e_7$. Then with respect to the basis, the matrix of $\lambda_S$ is
$$\begin{pmatrix} aI_8 &X\\ {}^t X& bI_8\end{pmatrix}\in M_{16}(K), \quad X=({\rm Tr}(e_i ((-e_j)\bar u)))_{1\le i,j\le 8}.
$$
Then the discriminant of the bilinear form is the determinant of the above matrix, which is given by
${\rm disc}(\lambda_S)=\det(ab I_8-{}^t X X).$
Now we claim that ${}^t XX$ is a diagonal matrix. Clearly, for each $j$, we have 
$$\sum_{k=0}^7 ({\rm Tr}(e_k(-e_j)\bar u))^2=N(u).
$$
Let $i\ne j$, and consider
\begin{equation}\label{Tr}
\sum_{k=0}^7 ({\rm Tr}(e_k (-e_i))\bar u))({\rm Tr}(e_k (-e_j))\bar u)).
\end{equation}

For a given $e_k$, let $e_k (-e_i)=e_l$ and $e_k (-e_j)=e_{l'}$. Then we claim that
there exists $e_a$ such that $e_a(-e_i)=e_{l'}$ and $e_a(-e_j)=-e_l$. This implies that $(\ref{Tr})=0$. 
Now, from $e_le_i=e_{l'}e_j$, we have
$$e_l=(-e_je_{l'})(-e_i)=(e_je_{l'})e_i=-e_j(e_{l'}e_i).
$$
by non-associativity. So $-e_le_j=e_je_l=e_{l'}e_i$. Let $e_a=e_l(-e_j)=e_{l'}e_i$.
Therefore, 
$\disc(\lambda_S)=\det(S)^8.$
\end{proof}

In order to prove the orthogonal relation in the proof of Lemma 5.9, by the above lemma, we need to consider the following:
Let $n$ be a positive integer and $T$ be a positive definite symmetric matrix of size $n$. 
Assume $T=(t_{ij})_{1\le i\le j\le n}$ is even integral, i.e., $t_{ii}\in \Z$ for $i=1,\ldots,n$ and 
$t_{ij}\in \frac{1}{2}\Z$ for $1\le i<j\le n$. 
For $\lambda\in \Q^n$, we define the theta function on $\mathbb{H}\times\C^n$ by 
$$\theta_{[\lambda]}(T;\tau,z)=\sum_{x\in\Z^n}\bold{e}({}^t(x+\lambda)T(x+\lambda)\tau+2{}^t(x+\lambda)Tz),\ 
(\tau,z)\in\mathbb{H}\times\C^n,\ \bold{e}(\ast)=e^{2\pi\sqrt{-1}\ast},
$$ 
where $[\lambda]$ stands for the image of $\lambda$ under the natural projection $\Q^n\lra \Q^n/\Z^n$ and 
the definition of the above theta function depends only on $[\lambda]$. 
Let $\Lambda_T$ be a complete representative of $(2T)^{-1}\Z^n/\Z^n$.  
\begin{lemma}\label{ortho} For any $\lambda,\mu\in \Lambda_T$, the following orthogonal relation holds: 
$$\ds\int_{(\C/\Z+\tau\Z)^n}
 \theta_{[\lambda]}(T;\tau,z)\overline{\theta_{[\mu]}(T;\tau,z)}e^{-4\pi({\rm Im} \tau)^{-1}
T[{\rm Im} z]}dz
=\left\{
\begin{array}{ll}
2^{-n}\det(T)^{-\frac{1}{2}}({\rm Im} \tau)^{\frac{n}{2}} & {\rm if}\ \lambda=\mu \\
0& {\rm otherwise}.
\end{array}\right.
$$ 
\end{lemma}
\begin{proof} Put $z=a+\tau b,\ a,b\in \R^n$. Then we have 
\begin{eqnarray*}
&&\ds\int_{(\C/\Z+\tau\Z)^n}
 \theta_{[\lambda]}(T;\tau,z)\overline{\theta_{[\mu]}(T;\tau,z)}e^{-4\pi({\rm Im} \tau)^{-1}
T[{\rm Im}z]}dz=\\
&&({\rm Im}\tau)^n\ds\int_{(\R/\Z)^n} \sum_{x,y\in\Z^n}\left\{\ds\int_{(\R/\Z)^n}
\bold{e}(2{}^t(x+\lambda)Ta-2{}^t(y+\mu)Ta)da\right\}\cdot \\
&&\bold{e}(2\sqrt{-1}({}^t(x+\lambda)Tb+{}^t(y+\mu)Tb))\bold{e}({}^t(x+\lambda)T(x+\lambda)\tau-{}^t(y+\mu)T(y+\mu)\overline{\tau}) e^{-4\pi({\rm Im} \tau)^{-1}T[{\rm Im}z]}\,db,
\end{eqnarray*}
where $T[{\rm Im}z]={}^t({\rm Im}z)T({\rm Im}z)$. 
Note that for given $x,y\in\Z^n$, $2 {}^t(x-y+\lambda-\mu)T\in \Z^n$ if and only if $\lambda=\mu$ by the definition.  
Therefore
$$\ds\int_{(\R/\Z)^n}
\bold{e}(2{}^t(x+\lambda)Ta-2{}^t(y+\mu)Ta)\, da
=\left\{
\begin{array}{ll}
1 & {\rm if}\ x=y\ {\rm and}\ \lambda=\mu \\
0& {\rm otherwise}.
\end{array}\right.
$$ 
If $\lambda=\mu$, we have 
\begin{eqnarray*}
&&\ds\int_{(\C/\Z+\tau\Z)^n}
 \theta_{[\lambda]}(T;\tau,z)\overline{\theta_{[\lambda]}(T;\tau,z)}e^{-4\pi({\rm Im} \tau)^{-1}
T[{\rm Im}z]}\, dz=\\
&&({\rm Im}\tau)^n\ds\int_{(\R/\Z)^n} \sum_{x\in\Z^n}e^{-4\pi ({\rm Im}\tau){}^t(b+x+\lambda)T(b+x+\lambda)}\, db
=({\rm Im}\tau)^n\ds\int_{\R^n} e^{-4\pi({\rm Im}\tau){}^tbTb}\, db.
\end{eqnarray*}
Since $T$ is diagonalizable by an orthogonal matrix over $\R$, we may assume that $T=\diag(t_1,\ldots,t_n),\ t_i\in \R_{>0}$. 
Hence
$$\ds\int_{\R^n} e^{-4\pi({\rm Im}\tau){}^tbTb}\, db=\prod_{i=1}^n\int_{\R}e^{-4\pi({\rm Im}\tau) t_i t^2}\, dt=
\prod_{i=1}^n\frac{1}{\sqrt{4({\rm Im}\tau) t_i}}=2^{-n}\det(T)^{-\frac{1}{2}}({\rm Im}\tau)^{-\frac{n}{2}}.
$$
Hence we have the claim. 
\end{proof}


\begin{thebibliography}{99}

\bibitem{B} W.L. Baily Jr., {\em An exceptional arithmetic group and its Eisenstein series}, Ann. of Math. {\bf 91} (1970), 512-549.
\bibitem{BB} W.L. Baily Jr. and A. Borel, {\em Compactification of arithmetic quotients of bounded symmetric domain}, Ann. of Math. {\bf 84} (1966), 442-528.
\bibitem{borel&jacquet} A. Borel and H. Jacquet, {\em Automorphic forms and automorphic representations}, Proc. Sympos. Pure Math., XXXIII,  1977, Part 1, 189-207.
\bibitem{Bour} N. Bourbaki, Groupes et Alg\`ebres de Lie, Chapitres 4,5 et 6, Paris, 1981.
\bibitem{casselman} W. Casselman, Introduction to Admissible Representations of $p$-adic Groups, available at 
http://www.math.ubc.ca/~cass/  
\bibitem{CM} J. Cogdell and P. Michel, {\em On the complex moments of symmetric power L-functions at $s = 1$}, Int. Math. Res. Not. 2004, no. 31, 1561-1617.
\bibitem{cog&ps} J. Cogdell and I. Piatetski-Shapiro, {\em Base change for the Saito-Kurokawa representations of $PGSp(4)$}, J. Number Th. {\bf 30} (1988), 298-320.
\bibitem{Coxeter} H.S.M. Coxeter, {\em Integral Cayley numbers}, Duke Math. J.  {\bf 13} (1946). 561-578. 
\bibitem{Ca} R.W. Carter, Simple groups of Lie type, Pure and Applied Mathematics, Vol. 28. John Wiley \& Sons, 1972.
\bibitem{F} H. Freudenthal, {\em Zur ebenen Oktavengeometrie}, Nederl. Akad. Wetensch. 
Proc. Ser. A. 56=Indagationes Math. {\bf 15} (1953), 195-200.
\bibitem{Ik3} T. Ikeda , {\em On the theory of Jacobi forms and Fourier-Jacobi coefficients of Eisenstein series}, J. Math. Kyoto Univ. 
{\bf 34} (1994), 615--636.
\bibitem{Ik1} \bysame, {\em On the lifting of elliptic cusp forms to Siegel cusp forms of degree $2n$}, Ann. of Math. {\bf 154} (2001), 641--681.
\bibitem{Ik2} \bysame, {\em On the lifting of Hermitian modular forms}, Comp. Math. {\bf 144} (2008), 1107-1154. 
\bibitem{Ik4} \bysame, {\em On the lifting of automorphic representation of $PGL_2(\A)$ to $Sp_{2n}(\A)$ of $\widetilde{Sp_{2n+1}(\A)}$ over a totally real algebraic field}, preliminary version.
\bibitem{Ik&H} T. Ikeda and K. Hiraga, {\em On the Kohnen plus space for Hilbert modular forms of half-integral weight I}, Comp. Math.  {\bf 149} (2013),  no. 12, 1963-2010.
\bibitem{Ka} M. Karel, {\em Fourier coefficients of certain Eisenstein series}, Ann. of Math. {\bf 99} (1974), 176-202. 
\bibitem{Ka1} \bysame, On certain Eisenstein series and their Fourier coefficients, Ph.D. thesis, 1972.
\bibitem{kim} H. Kim, {\em Exceptional modular form of weight 4 on an exceptional domain contained in $\Bbb C^{27}$}, Rev. Mat. Iberoamericana {\bf 9} (1993), 139-200.
\bibitem{kim1} \bysame, {\em On local $L$-functions and normalized intertwining operators}, Can. J. Math. {\bf 57} (2005), 535--597. 
\bibitem{knapp} A.W. Knapp, Representation Theory of Semisimple Groups, Princeton University Press, 1986.
\bibitem{Krieg} A. Krieg, {\em Jacobi forms of several variables and the Maass space}, J. Number Theory {\bf 56} (1996), 242-255.
\bibitem{kudla-split} S. Kudla, {\em Splitting metaplectic covers of dual reductive pairs}, Israel J. Math. {\bf 87} (1994), 361--401.
\bibitem{Ku} \bysame, {\em Some extensions of the Siegel-Weil formula}, Eisenstein series and applications, 205-237, Progr. Math., 258, Birkh\"auser, Boston, MA, 2008.
\bibitem{kuro} N. Kurokwa, {\em Examples of eigenvalues of Hecke operators on Siegel cusp forms of degree two}, Inv. Math. {\bf 49} (1978), 149--165.
\bibitem{langlands}R. Langlands, On the functional equations satisfied by Eisenstein series. Lecture Notes in Mathematics, Vol. 544. Springer-Verlag, Berlin-New York, 1976. v+337 pp.
\bibitem{steinberg} R. Steinberg, Lectures on Chevalley groups. Notes prepared by John Faulkner and Robert Wilson. Yale University, New Haven, Conn., 1968.
\bibitem{ps}I. I. Piatetski-Shapiro, {\em On the Saito-Kurokawa lifting}, Invent. Math. {\bf 71} (1983), 309--338. 
\bibitem{takase} K. Takase, {\em On two-fold covering group of $Sp(n,\R)$ and automorphic factor of weight 1/2},
Comment Math. Univ. St. Paul. {\bf 45} (1996), 117--145. 
\bibitem{Yamana} S. Yamana, {\em On the lifting of elliptic cusp forms to cusp forms on quaternionic unitary groups},  
J. Number Th. {\bf 130} (2010), 2480--2527.
\bibitem{W} M. Weissman, {\em The Fourier-Jacobi map and small representations}, Representation Theory, {\bf 7} (2003), 275--299.
\end{thebibliography}
\end{document}